\documentclass[a4paper,11pt,reqno]{amsart}
\usepackage{amsmath,amssymb,amsfonts,amscd}
\usepackage{hyperref,verbatim}
\usepackage{mathrsfs,ytableau,array,paralist,enumitem,xcolor}
\usepackage[all]{xy}
\usepackage{graphics}

\theoremstyle{plain}
  \newtheorem{thm}{Theorem}[section]
  \newtheorem{cor}[thm]{Corollary}
  \newtheorem{lem}[thm]{Lemma}
  \newtheorem{prop}[thm]{Proposition}
\theoremstyle{definition}
  \newtheorem{defn}[thm]{Definition}
\theoremstyle{remark}
  \newtheorem{rem}[thm]{Remark}
  \newtheorem{eg}[thm]{Example}
  
\numberwithin{equation}{section}

 \newcommand{\ZZ}{\mathbb{Z}}

 \newcommand{\QQ}{\mathbb{Q}}
 
 \newcommand{\ZP}{\mathbb{Z}_{(\emph{p})}}
 \newcommand{\sgn}{\mathrm{sgn}}

 \DeclareMathOperator{\Tab}{\mathcal{T}}
 \DeclareMathOperator{\RSTab}{\mathrm{RStd}}                          
 \DeclareMathOperator{\SSTab}{\mathrm{SStd}}                          
 \DeclareMathOperator{\STab}{\mathrm{Std}}                          
 \DeclareMathOperator{\wt}{wt}

 \newcommand{\sym}[1]{\mathfrak{S}_{#1}}
 \newcommand{\s}{\mathfrak{s}}
 \newcommand{\rr}{\mathfrak{r}}
 \renewcommand{\t}{\mathfrak{t}}
 \newcommand{\uu}{\mathfrak{u}}
 \newcommand{\vv}{\mathfrak{v}}
\newcommand{\IT}{\t}

 \newcommand{\dd}{\mathsf{d}}
\newcommand{\Shape}{\mathrm{Shape}}

\newcommand{\res}{\mathrm{res}}
\newcommand{\bs}{\mathsf{s}}
\newcommand{\BW}{\mathbf{W}}
\newcommand{\bw}{\mathbf{w}}
\newcommand{\lcm}{\mathrm{lcm}\,}

\newcommand{\upt}[2]{{#1} {\uparrow^{#2}}}
\newcommand{\upI}[2]{\IT^{#1} {\uparrow^{#2}}}
\newcommand{\upa}[2]{\IT^{#1|#2}}

\newcommand{\kl}{\lambda}
\newcommand{\klm}{\lambda+(m)}
\newcommand{\Ew}[1]{E_{#1}}

\newcommand{\FM}[1]{\textcolor{red}{#1}}
\newcommand{\KM}[1]{\textcolor{violet}{#1}}
\newcommand{\KJ}[1]{\textcolor{blue}{#1}}
\newcommand{\DD}[1]{D(#1)}

\newcommand{\ba}{b}
\newcommand{\add}[2]{{#1} \oslash {#2}}

\renewcommand{\lhd}{\vartriangleleft}
\renewcommand{\rhd}{\vartriangleright}

\renewcommand{\unrhd}{\trianglerighteq}

\allowdisplaybreaks

\setlist{left=0pt, itemsep=6pt}
\setlist[enumerate,1]{left=0pt, itemsep=6pt, label= {\normalfont {(\arabic*)}}}
\setlist[enumerate,2]{left=0pt, itemsep=6pt, label= {\normalfont {(\alph*)}}}

\begin{document}

\title[Young's seminormal basis vectors and their denominators]{Young's seminormal basis vectors \\and their denominators}
\author{Ming Fang}
\address[M. Fang]{HLM, HCMS, Academy of Mathematics and Systems Science, Chinese Academy of Sciences, Beijing, 100190 -and- School of Mathematical Sciences, University of Chinese Academy of Sciences, Beijing, 100049, People's Republic of China.}
\email{fming@amss.ac.cn}

\author{Kay Jin Lim}
\address[K. J. Lim]{Division of Mathematical Sciences, Nanyang Technological University, SPMS-04-01, 21 Nanyang Link, Singapore 637371.}
\email{limkj@ntu.edu.sg}

\author{Kai Meng Tan}
\address[K. M. Tan]{Department of Mathematics, National University of Singapore, Block S17, 10 Lower Kent Ridge Road, Singapore 119076.}
\email{tankm@nus.edu.sg}

\thanks{The first author is supported by Natural Science Foundation of China (No.\ 11471315, 11321101 and 11688101), the second author is supported by Singapore MOE AcRF RG17/20 and the third author is supported by Singapore MOE AcRF R-146-000-317-114.}

\maketitle

\begin{abstract}
We study Young's seminormal basis vectors of the dual Specht modules of the symmetric group, indexed by a certain class of standard tableaux, and their denominators.  These vectors include those whose denominators control the splitting of the canonical morphism $\Delta(\lambda+\mu) \to \Delta(\lambda) \otimes \Delta(\mu)$ over $\ZP$, where $\Delta(\nu)$ is the Weyl module of the classical Schur algebra labelled by $\nu$.
\end{abstract}

\section{Introduction} \label{sec:introduction}
Let $n$ be a positive integer.  It is well known that the dual Specht modules $S^{\QQ}_{\lambda}$, as $\lambda$ runs over all partitions of $n$, give a complete set of irreducible modules of $\QQ\sym{n}$.  There are two distinguished bases for each $S^{\QQ}_{\lambda}$, namely the standard basis and Young's seminormal basis, both indexed by the set of standard $\lambda$-tableaux.  These two bases play a significant role in the study of
the representation theory of symmetric groups; see for example \cite{Mathasbook99} and the references therein.

While the transition matrix between these two bases is unitriangular, its entries in general are rational numbers and not integers.
Although the off-diagonal entries of the transition matrix can be computed recursively, we are not aware of any work that has been done to determine a closed formula for any of these entries.
For a standard $\lambda$-tableau $\t$, the denominator of Young's seminormal basis vector $f_{\t}$, denoted $\dd_{\t}$, is the least positive integer $k$ such that $kf_{\t}$ lies in the $\ZZ$-span of the standard basis.
This is of course the least (positive) common multiple of the denominators appearing in the row labelled by $\t$ of the transition matrix from the standard basis to Young's seminormal basis.

Young's seminormal basis controls the modular representation theory of symmetric groups in many ways; see for example \cite{FLT, Mathasbook99, RHansen10, RHansen13}.  Natural questions related to the arithmetical properties of Young's seminormal basis vectors (such as their denominators) arise. As far as we know, such knowledge is scant in the available literature, but is expected to connect with the other parts of the modular representation theory \cite{FLT,RHansen10}.

Indeed, the main motivation of the work presented here is \cite{FLT}, in which the authors initiated a study into comparing the Jantzen filtrations of Weyl modules for a semisimple algebraic group $G$ over an algebraically closed field $k$ of characteristic $p>0$.
They showed that when the canonical $G$-morphism $\iota_{\lambda,\mu} : \Delta(\lambda+ \mu) \to \Delta(\lambda) \otimes \Delta(\mu)$ splits over $\ZP$, the localised ring of $\ZZ$ at the prime ideal $(p)$, then the Jantzen filtration of $\Delta(\lambda)$ may be naturally `embedded' into that of $\Delta(\lambda+\mu)$ (see \cite[Theorem 3.1]{FLT}).
This led to a detailed study of the split condition of $\iota_{\lambda,\mu}$ when $G$ is of type $A$, which was shown to be equivalent to a condition in terms of $\theta_{\lambda,\mu}$, the greatest common divisor of the coefficients of the product of certain Young symmetrizers associated to $\lambda$ and $\mu$, as well as a condition in terms of the denominator $\dd_{\add{\IT^{\lambda}}{\IT^{\mu}}}$ of $f_{\add{\IT^{\lambda}}{\IT^{\mu}}}$ when the last column of the Young diagram $[\lambda]$ is no shorter than the first column of the Young diagram $[\mu]$
(see \cite[Section 2.3]{FLT} for the definition of $\add{\s}{\t}$ for general tableaux $\s$ and $\t$).
By \cite[Theorem 3.13]{FLT}, the determination of $\theta_{\lambda,\mu}$ is equivalent to that of $\dd_{\add{\IT^{\lambda}}{\IT^{\mu}}}$.
The examples in \cite[Section 4]{FLT} show that $\theta_{\lambda,\mu}$ is very difficult to compute in general.

In this paper, we develop some techniques to study the $f_{\add{\IT^{\lambda}}{\IT^{\mu}}}$'s mentioned above and compute its denominator $\dd_{\add{\IT^{\lambda}}{\IT^{\mu}}}$ instead.
In fact, we study $f_{\upI{\lambda}{\nu}}$ for partitions $\lambda$ and $\nu$ such that the Young diagram $[\nu]$ contains the Young diagram $[\lambda]$.
Here, $\upI{\lambda}{\nu}$ is the largest standard $\nu$-tableau that contains the initial $\lambda$-tableau $\IT^{\lambda}$ as a subtableau, and $\add{\IT^{\lambda}}{\IT^{\mu}} = \upI{\lambda}{\lambda+\mu}$ when the last column of $[\lambda]$ is no shorter than the first column of $[\mu]$.
This class of Young's seminormal basis elements has the following very nice property: each is spanned by standard basis elements labelled by tableaux which are colour-semistandard, and those with the same colour type have the same coefficients (see Definition \ref{defn:colour} and Theorem \ref{lem:semistandard}).

Our first main result (Theorem \ref{thm:mu=(1)}) is the closed formula for $f_{\add{\IT^{\lambda}}{\IT^{(1)}}}$ in terms of the standard basis of $S^{\QQ}_{\lambda+(1)}$.
This result may be considered as the counterpart of \cite[Theorem 1.2]{Raicu14}, which provides a simplified way of computing the product of certain Young symmetrizers, 
from which $\theta_{\lambda,(1)}$ could possibly be deduced.
We note in addition that Theorem \ref{thm:mu=(1)} may also be derived from \cite[Theorem 1]{RHansen10} 
in the context of Iwahori-Hecke algebras using a different inductive approach.

We next study $f_{\add{\IT^{(k,\ell^s)}}{\IT^{(m)}}}$.  We provide closed formulae for $f_{\add{\IT^{(k,\ell^s)}}{\IT^{(m)}}}$, and hence $\dd_{\add{\IT^{(k,\ell^s)}}{\IT^{(m)}}}$, in the cases $s=1$ and $\ell = 1$ (Theorem \ref{thm:klm} and Corollary \ref{cor:denominator-ell=1}) respectively, which can be used to determine $\theta_{(k,\ell),(m)}$ and $\theta_{(1^n),(m)}$ that have been computed in \cite[Section 4]{FLT}.
Readers who are familiar with \cite{FLT} will appreciate the succinctness and superiority of this new approach in computing these two numbers.
While we did not succeed in providing a closed formula for $f_{\add{\IT^{(k,\ell^s)}}{\IT^{(m)}}}$ in the general case, we are able to obtain some reduction results (Theorem \ref{thm:4reductions}).

Reduction actually holds in a more general setting, and using this, we obtain various upper bounds for $\dd_{\upI{\lambda}{\nu}}$.
We give some examples which show that these upper bounds are optimal in some cases.

Our results on the denominator $\dd_{\upI{\lambda}{\nu}}$ may be summarised as follows:

\begin{thm} \label{thm:summary}
Let $\lambda = (\lambda_1,\dotsc,\lambda_r)$, with $\lambda_r > 0$, and $\nu = (\nu_1,\dotsc, \nu_t)$ be partitions such that $[\lambda] \subseteq [\nu]$.

\begin{enumerate}
\item
If $\lambda$ is obtained from $\nu$ by removing a removable node $A$, and $B_1,\dotsc, B_s$ are the removable nodes of $\nu$ below $A$, then
    $$
    \dd_{\upI{\lambda}{\nu}} = \prod_{i=1}^s (\res(A) - \res(B_{i})),
    $$
    where $\res(C)$ denotes the residue of the node $C$.

\item If $r = 2$, then
\begin{align*}
  \dd_{\upI{\lambda}{\nu}} &=  \frac{\lcm[\lambda_1-\lambda_2+1,\, \lambda_1-\lambda_2+1 + \min(\lambda_2,\, \nu_1-\lambda_1)]}{\lambda_1-\lambda_2+1}.
\end{align*}
Here $[a,b] = \{ i \mid \ZZ \mid a \leq i \leq b \}$ for $a, b \in \ZZ$, and $\lcm S$ denotes the least common multiple of the elements in $S$ for $S \subseteq \ZZ^+$.

\item If $r\geq 2$ and $\lambda_2 = \dotsb = \lambda_r = \nu_2 = \dotsb = \nu_r$, then
$$\dd_{\upI{\lambda}{\nu}} = \dd_{\upI{(k,\ell^s)}{(k+\ell,\ell^s)}},$$
where
    $k = \max(\lambda_1-\lambda_2+r,\min(\lambda_1,\nu_1-\lambda_2))$,
    $\ell =  \min(\lambda_2,\nu_1-\lambda_1)$ and
    $s = \min(r, \ell).$

\item We have
\begin{alignat*}{2}
  \dd_{\upI{\lambda}{\nu}}
  &= \dd_{\upI{\lambda}{(\nu_1,\dotsc,\nu_{r-1},\lambda_r)}}; \\
  \dd_{\upI{\lambda}{\nu}}
  &= \dd_{\upI{(\lambda_2,\dotsc, \lambda_r)}{(\nu_2,\dotsc, \nu_t)}}& \quad &\text{if $\lambda_1 = \nu_1$ and $r \geq 2$}.
\end{alignat*}

\item For all positive integers $m$ with $m \leq \nu_1 - \lambda_1$, we have
  $$\dd_{\upI{\lambda}{\nu}}
  \mid \dd_{\upI{\lambda}{\lambda+(m)}} \dd_{\upI{\lambda+(m)}{\nu}}.$$

\item For all positive integers $m$ and $i$ with $2 \leq i \leq r-1$, we have
$$  \dd_{\upI{\lambda}{\lambda+(m)}}
  \mid \dd_{\upI{\lambda^{(i)}}{\lambda^{(i)}+(m)}} \dd_{\upI{\lambda^{\leqslant i}}{\lambda^{\leqslant i}+(m)}},$$
  where $\lambda^{(i)} =   (\lambda_1+i-1,\lambda_{i+1},\dotsc,\lambda_r)$ and $\lambda^{\leqslant i} = (\lambda_1,\dotsc, \lambda_i)$.

\end{enumerate}
\end{thm}

Parts (1) and (2) of Theorem \ref{thm:summary} provide closed formulae for $\dd_{\upI{\lambda}{\nu}}$ for specific $\nu$ and $\lambda$ respectively, while parts (3)--(6) relate $\dd_{\upI{\lambda}{\nu}}$ to `smaller' denominators, and hence may be used to provide upper bounds for general $\dd_{\upI{\lambda}{\nu}}$ not covered in parts (1) and (2).

We now indicate the organisation of this paper.  After providing the necessary background in the next section, we look at the seminormal basis vector $f_{\add{\IT^{\lambda}}{\IT^{(1)}}}$ in Section \ref{sec:mu=(1)}, 
and $f_{\add{\IT^{(k,\ell^s)}}{\IT^{(m)}}}$ in Section \ref{sec:kl^sm}.
In our concluding Section \ref{sec:general}, we relate $f_{\upI{\lambda}{\nu}}$ and its denominator $\dd_{\upI{\lambda}{\nu}}$ to $f_{\upI{\tilde{\lambda}}{\tilde{\nu}}}$ and $\dd_{\upI{\tilde{\lambda}}{\tilde{\nu}}}$ for some smaller $\tilde{\lambda}$ and $\tilde{\nu}$.

\begin{rem}\hfill
\begin{enumerate}
\item As our motivation from \cite{FLT} is to study $\dd_{\add{\IT^{\lambda}}{\IT^{\mu}}}$ for the symmetric groups, we choose to present our work in this context here.  We believe that most, if not all, of our results should generalise to Iwahori-Hecke algebras without much difficulty.
\item
Our approach to Young's seminormal basis vectors is different from \cite{RHansen10}.
We study directly $\DD{\t} = \sum_{\s \in \STab(\lambda)} q_{\t,\s} d(\s)$ (see Definition \ref{defn:D}), which is a distinguished element of $\QQ \sym{|\lambda|}$ satisfying $f_{\t} = \DD{\t} e_{\IT^{\lambda}}$.  With the introduction of our key notion of colour-semistandardness (Definition \ref{defn:colour}), we are able to obtain closed formulae for $\DD{\upI{\lambda}{\nu}}$ in some cases, and relate the coefficients $q_{\upI{\lambda}{\nu},\s}$ to those coming from smaller partitions in some others.
\end{enumerate}

\end{rem}

\section{Preliminaries} \label{sec:preliminaries}

In this section, we recall the background theory and prove some preliminary results.  For a large part, we follow the notations that have been used in \cite{FLT}.

Throughout this paper, we use the following notation, for $a,b \in \ZZ$:
$$
[a,b] := \{ k \in \ZZ \mid a \leq k \leq b \}.
$$
Also, for $S \subseteq \ZZ^+$, $\lcm S$ denotes the least common multiple of the integers in $S$.

\subsection{Symmetric groups}
Denote the group of bijections on a nonempty set $X$ by $\sym{X}$.  We view elements of such a group as functions, so that we compose these elements from right to left. When $Y$ is a nonempty subset of $X$, we view $\sym{Y}$ as a subgroup of $\sym{X}$ by identifying an element of $\sym{Y}$ with its extension that sends $x$ to $x$ for all $x \in X \setminus Y$.

Let $X \subseteq \ZZ$ and $k \in \ZZ$.  Define $X^{+k} := \{ x+k \mid x \in X \}$, and for any function $\sigma : X \to X$, write $\sigma^{+k} : X^{+k} \to X^{+k}$ for the function such that $\sigma^{+k}(x+k) = \sigma(x) + k$ for all $x \in X$.  Then $\sigma \mapsto \sigma^{+k}$ is a group isomorphism from $\sym{X}$ to $\sym{X^{+k}}$, and this extends further to give an isomorphism $\QQ\sym{X} \to \QQ\sym{X^{+k}}$.  If $R \subseteq \QQ\sym{X}$, we write $R^{+k}$ for $\{ r^{+k} \mid r\in R \}$.  In particular, $\sym{X}^{+k} = \sym{X^{+k}}$.

Let $n \in \ZZ^+$, the set of all positive integers.  We write $\sym{n}$ for $\sym{[1,n]}$.  It is well known that $\sym{n}$ is a Coxeter group with the basic transpositions $\bs_i := (i,i+1)$, one for each $i \in [1,n-1]$, as its Coxeter generators.

\subsection{Compositions, partitions and Young tableaux}
\label{subsec:combinatorics-Young-tableaux}
A composition $\lambda = (\lambda_1,\lambda_2,\dotsc)$ is a sequence of non-negative integers which are eventually zero.  We write $|\lambda|$ for $\sum_{i=1}^{\infty} \lambda_i$.  If $|\lambda| = n$, we say that $\lambda$ is a composition of $n$, and write $\lambda \vDash n$.  The Young subgroup $\sym{\lambda}$ is
$$
\sym{\lambda} = \sym{\lambda_1} \sym{\lambda_2}^{+\lambda_1} \sym{\lambda_3}^{+(\lambda_1+\lambda_2)} \dotsm \subseteq \sym{n}.
$$
This is a parabolic subgroup of $\sym{n}$ as a Coxeter group.
The dominance order $\unrhd$ on all compositions is given by: $\lambda \unrhd \mu $ if and only if $\lambda_1+ \cdots+\lambda_k\geq \mu_1+\cdots+\mu_k$ for all $k\in \ZZ^+$.

Let $\lambda \vDash n$.  The Young diagram of $\lambda$ is defined to be the set $ [\lambda] = \{(a,b) \in (\mathbb{Z}^+)^2 \mid b\leq \lambda_a\};$ and we call its elements the {\em nodes} of $\lambda$. Following \cite[3.30]{Mathasbook99}, for a node $(a,b) \in [\lambda]$, its {\em residue} $\res((a,b))$ is defined as $\res((a,b)) = b-a$.  We depict $[\lambda]$ as an array of left-justified boxes in which the $i$-th row comprises exactly $\lambda_i$ boxes, with each box representing a node of $\lambda$.

A $\lambda$-tableau is a bijective map $\s: [\lambda]\to [1, n]$, in which case $\lambda$ is said to be the shape of $\s$, denoted by $\Shape(\s)$.  We identify $\s$ with the pictorial depiction of the Young diagram $[\lambda]$ in which each box in $[\lambda]$ is filled with $[1, n]$ so that each integer appears exactly once. When $\s(i,j) = k$, the residue of $k$ in $\s$, denoted $\res_{\s}(k)$, is $\res((i,j))$.  Denote the set of all $\lambda$-tableaux by $\Tab(\lambda)$.

A $\lambda$-tableau $\s$ is said to be \emph{row standard} (respectively, {\em column standard}) if its entries are increasing along each row (respectively, down each column). If $\t \in \Tab(\lambda)$, we write $\overline{\t}$ for the row standard $\lambda$-tableau obtained by rearranging the entries in each row of $\t$.  Let $\RSTab(\lambda)$ be the set of all row standard $\lambda$-tableaux.
A $\lambda$-tableau is {\em standard} if it is both row and column standard, and we denote the set of all standard $\lambda$-tableaux by $\STab(\lambda)$.

Let $\s \in \RSTab(\lambda)$ and $r \in [1, n]$.  Since $\s$ is row standard, $\s^{-1}([1,r])$ is the Young diagram of a composition, and we define the subtableau $\s{\downarrow_r}$ of $\s$ to be the restriction of $\s$ to this subdomain.  Pictorially, $\s{\downarrow_r}$ consists precisely of those boxes in $[\lambda]$ which are filled with $1,\dotsc,r$ by $\s$. The dominance order $\unrhd$ on $\RSTab(\lambda)$ is given by $\s\unrhd \t$ if and only if, for each $r \in [1, n]$, we have \[\Shape(\s{\downarrow_r})\unrhd  \Shape(\t{\downarrow_r}).\]
We write $\s\rhd\t$ and $\t \lhd \s$ if $\s\unrhd\t$ and $\s\neq\t$.

Now suppose further that $\lambda_1 \geq \lambda_2 \geq \dotsb$.  In this case, we call $\lambda$ a partition of $n$, denoted $\lambda \vdash n$.  In this paper, all partitions are nonempty (but we allow the composition $(0,0,\dotsc)$), and we write $\lambda=(\lambda_1,\dotsc,\lambda_r)$ where $r = \max\{i \in \mathbb{Z}^+\mid \lambda_i >0 \}$.
A node $(a,b) \in [\lambda]$ is {\em removable} if $(a+1,b), (a,b+1) \notin [\lambda]$.

\begin{defn} \label{defn:uparrow}
Let $\lambda \vdash n $.
\begin{enumerate}
\item Let $\nu \vdash m$ such that $[\lambda] \subseteq [\nu]$.  For $\s \in \STab(\lambda)$, we define $\upt{\s}{\nu}$ to be the standard $\nu$-tableau where $(\upt{\s}{\nu}){\downarrow_{n}} = \s$ and the nodes of $\upt{\s}{\nu}$ lying in the skew Young diagram $[\nu] \setminus [\lambda]$ are filled with $[n+1,m]$ in turn, starting with the top row, going from left to right in each row, and down the rows.

\item
The {\em initial $\lambda$-tableau}, denoted $\IT^{\lambda}$, is $\upt{\t_0}{\lambda}$, where $\t_0$ is the unique (standard) $(1)$-tableau.

\item To ease the notation, for $m \in \ZZ^+$, we write $\upa{\lambda}{m}$ for $\upI{\lambda}{\lambda+(m)}$. Here, and hereafter, $\lambda + (m) = (\lambda_1+m,\lambda_2, \dotsc, \lambda_r)$ when $\lambda = (\lambda_1,\dotsc, \lambda_r)$.
\end{enumerate}
\end{defn}

\begin{rem}
  Note that  $\upt{\s}{\nu}$ is the largest (with respect to $\unrhd$) row standard $\nu$-tableau $\t$ such that $\t {\downarrow_n} = \s$, and that $\IT^{\lambda}$ is the largest row standard $\lambda$-tableau.
\end{rem}

We illustrate Definition \ref{defn:uparrow} with the following example:
$$
\upI{(2,2)}{(4,3,2)} =
\raisebox{3mm}{
\ytableausetup{boxsize=1em, mathmode}
\begin{ytableau}
\scriptstyle 1 & \scriptstyle 2  & \scriptstyle \FM{5} &\scriptstyle \FM{6} \\
\scriptstyle 3 & \scriptstyle 4  & \scriptstyle \FM{7} \\
\scriptstyle \FM{8} & \scriptstyle \FM{9}
\end{ytableau}}
$$

\begin{lem} \label{lem:t^lambda-uparrow^nu}
Let $\lambda \vdash n$ and $\nu \vdash m$ with $[\lambda] \subseteq [\nu]$.  If $\s \in \STab(\nu)$ such that $\s \rhd \upI{\lambda}{\nu}$, then $\Shape(\s{\downarrow_n}) \rhd \lambda$.
\end{lem}

\begin{proof}
If $\s \rhd \upI{\lambda}{\nu}$, then $\Shape(\s {\downarrow_r}) \unrhd \Shape((\upI{\lambda}{\nu}) {\downarrow_r})$ for all $r \in [1,m+n]$.
If the inequality is not strict at $r=n$, then $\Shape(\s {\downarrow_n}) = \Shape((\upI{\lambda}{\nu}) {\downarrow_n}) = \Shape(\IT^{\lambda}) = \lambda$, and so $\rr := \s{\downarrow_{n}} \in \STab(\lambda)$.
For each $r \in [1,n]$, we have
$$
\Shape(\rr {\downarrow_r}) = \Shape(\s{\downarrow}_r) \unrhd  \Shape((\upI{\lambda}{\nu}) {\downarrow_r}) = \Shape(\IT^{\lambda} {\downarrow_r}),
$$
so that $\rr \unrhd \IT^{\lambda}$, and hence $\rr = \IT^{\lambda}$ since $\IT^{\lambda}$ is the largest standard $\lambda$-tableau.  But then $\s \rhd \upI{\lambda}{\nu}$ and $\s{\downarrow_n} = \rr = \IT^{\lambda}$ contradict the maximality of $\upI{\lambda}{\nu}$.
\end{proof}

Post-composition of $\lambda$-tableaux by elements of $\sym{n}$ gives a well-defined, faithful and transitive left action of $\sym{n}$ on $\Tab(\lambda)$, i.e.\ $\sigma \cdot \s = \sigma \circ \s$ for $\sigma \in \sym{n}$ and $\s\in \Tab(\lambda)$.  For a $\lambda$-tableau $\s$, we denote the row and column stabilizers of $\s$ under this action by $R_{\s}$ and $C_{\s}$, respectively.  In addition, we write $d(\s)$
for the unique element in $\sym{n}$ such that $\s= d(\s) \cdot \IT^\lambda$, or equivalently $d(\s) =  \s \circ (\IT^{\lambda})^{-1}$.

We shall require the following elementary result about standard $\lambda$-tableaux.

\begin{lem} \label{lem:std}
Let $\lambda \vdash n$.
\begin{enumerate}
\item
If $w \in \sym{n}$ and $\t \in \STab(\lambda)$ such that $w \cdot \t \in \STab(\lambda)$, then $w$ has a reduced expression $w= \bs_{i_{\ell}} \bs_{i_{\ell -1}} \dotsm \bs_{i_1}$ such that $(\bs_{i_j}\cdots \bs_{i_2}\bs_{i_1})\cdot\t \in \STab(\lambda)$ for every $j\in [1,\ell-1]$.

\item
If $\t \in \STab(\lambda)$ and $(i+1,j) \in [\lambda]$ with $i >0$, then there exists $w \in \sym{[\t(i,j),\, \t(i+1,j)]}$ such that $w\cdot \t \in \STab(\lambda)$ and $(w\cdot \t)(i+1,j) - (w\cdot \t)(i,j) =1$.
\end{enumerate}
\end{lem}

\begin{proof} \hfill
\begin{enumerate}
\item
  We prove by induction on the length $\ell(w)$ of $w$.  There is nothing to prove if $\ell(w) =0$.  For $\ell(w) >0$, we have $w \ne 1$, so that $w(j) > w(j+1)$ for some $j \in [1,n-1]$.  Since $\t$ and $w \cdot \t$ are standard, $j$ and $j+1$ cannot be lying in the same row or same column in $\t$, so that $\bs_j \cdot \t \in \STab(\lambda)$.  Now $\ell (w\bs_j) = \ell(w) -1$ \cite[1.4 Corollary]{Mathasbook99}, and so applying the induction hypothesis to $w{\bs_j}$ and $\bs_j \cdot \t$ finishes the proof.
\item
Let $a_{\t} = \t(i,j)$ and $b_{\t}= \t(i+1,j)$.  We prove by induction on $b_{\t} - a_{\t}$, where $w = 1$ if $b_{\t} - a_{\t}=1$.
For $b_{\t} - a_{\t} >1$, if there exists $a' \in [a_{\t}+1,b_{\t}-1]$ which does not lie in the $i$-th row of $\t$, then choose $a'$ to be the least such and let $\t' = (a_{\t},a_{\t}+1,\dotsc, a') \cdot \t$; otherwise let $\t' = (b_{\t}-1,b_{\t}) \cdot \t$.  Then $\t' \in \STab(\lambda)$ with $[a_{\t'} , b_{\t'}] = [\t'(i,j) , \t'(i+1,j)]  \subsetneq [a_{\t}, b_{\t}]$, and applying the induction hypothesis to $\t'$ finishes the proof.
\end{enumerate}
\end{proof}

\subsection{Dual Specht modules}  \label{subsec:Specht-and-Young-basis}
Let $\lambda$ be a partition of $n$.
We briefly review the construction of the signed permutation module $\tilde{M}^\lambda_{\ZZ}$ \cite[\S7.4]{Fulton97}.
Let $\ZZ\Tab(\lambda)$ be the free $\ZZ$-module with basis $\Tab(\lambda)$.
Then $\tilde{M}^\lambda_{\ZZ}$ is the quotient of $\ZZ\Tab(\lambda)$ by the relations $\gamma \cdot \t = \sgn(\gamma)\t$ for all $\t \in \Tab(\lambda)$ and $\gamma \in C_{\t}$. Let $[\t]$ denote the image of $\t$ in $\tilde{M}^{\lambda}_\ZZ$.  The left action of $\sym{n}$ on $\Tab(\lambda)$ induces a $\ZZ\sym{n}$-module structure on $\tilde{M}^{\lambda}_{\ZZ}$, with $\sigma[\t] = [\sigma \cdot \t]$ for all $\t \in \Tab(\lambda)$ and $\sigma \in \sym{n}$.

The {\em integral dual Specht module} $S^{\ZZ}_\lambda$ is the $\ZZ$-span of the polytabloids $e_{\t} := \sum_{\rho \in R_{\t}} \rho[\t]$.  This is actually a $\ZZ\sym{n}$-submodule of $\tilde{M}^{\lambda}_{\ZZ}$, since $\sigma e_{\t} = e_{\sigma \cdot \t}$ for all $\t \in \Tab(\lambda)$ and $\sigma \in \sym{n}$.
These polytabloids satisfy the following:
\begin{itemize}
\item If $\t \in \Tab(\lambda)$ and $\rho \in R_{\t}$, then
\begin{equation} \label{E:row-equiv}
e_{\t} = e_{\rho \cdot \t}.
\end{equation}
\item If $X$ and $Y$ are subsets of the $i$-th and $(i+1)$-th rows of $\t \in \Tab(\lambda)$ respectively with $|X\cup Y|>\lambda_i$ and $G_{X,Y}$ is a left transversal of $\sym{X}\sym{Y}$ in $\sym{X\cup Y}$,
then
\begin{equation} \label{E:Garnir-2}
\left(\sum_{\sigma\in G_{X,Y}} \sigma\right) e_\t=0.
\end{equation}
(A proof of this may be adapted from that of \cite[Theorem 7.2]{GJames}.)
\end{itemize}
The following relation then follows from \eqref{E:row-equiv} and \eqref{E:Garnir-2}:
If $\t \in \Tab(\lambda)$ and $X$ is a subset of its $(i+1)$-th row, then
\begin{equation}  \label{E:Garnir-1}
e_{\t} = (-1)^k \sum_{Y} e_{\t_Y}
\end{equation}
where $k = |X|$ and the sum runs over all $k$-element subsets $Y$ of the $i$-th row of $\t$, and $\t_Y$ is any $\lambda$-tableau obtained from $\t$ by interchanging $X$ and $Y$ (any such $\t_Y$ gives the same $e_{\t_Y}$). 

\begin{eg}
Consider the $(3,2)$-tableau \[\ytableausetup{notabloids}\t=
\raisebox{1mm}{
\begin{ytableau}
\scriptstyle 1 & \scriptstyle 4& \scriptstyle 5\\ \scriptstyle 2& \scriptstyle 3
\end{ytableau}
}
\, .\]
For this example, let $e_{i,j}$ denote the polytabloid containing the tabloid with the numbers $i<j$ in the second row, so that, for example, $e_\t=e_{2,3}$. 
Using \eqref{E:Garnir-2} with $X=\{4,5\}$ and $Y=\{2,3\}$, we have
\begin{align*}
  0&=(1+(2,3,4)+(2,3,5,4)+(3,4)+(3,5,4)+(2,4)(3,5))e_{2,3}\\
  &=e_{2,3}+e_{3,4}+e_{3,5}+e_{2,4}+e_{2,5}+e_{4,5},
\end{align*}
so $e_{\t} = e_{2,3} = - e_{3,4} - e_{3,5} -e_{2,4} -e_{2,5} -e_{4,5}$. 
\end{eg}

In particular, \eqref{E:row-equiv} and \eqref{E:Garnir-2} imply that the set $\{ e_{\t} \mid \t \in \STab(\lambda)\}$ of {\em standard polytabloids} is a basis---called the {\em standard basis}---for $S^{\ZZ}_{\lambda}$.  In addition, they also imply the following `straightening rules' which will be used in this paper.

\begin{prop} \label{prop:Garnir}
Let $\lambda \vdash n$, and let $\t \in \RSTab(\lambda) \setminus \STab(\lambda)$, say for $(i,j), (i+1,j) \in [\lambda]$ we have $a := \t(i,j) > \t(i+1,j) =: b$.
\begin{enumerate}
\item Let $X = \{ \t(i,s) \mid s \in [j, \lambda_i] \}$ and $Y = \{ \t(i+1,r) \mid r \in [1,j] \}$ (so $|X \cup Y| = \lambda_i+1$).  Pick a left transversal $G_{X,Y}$ of $\sym{X}\sym{Y}$ in $\sym{X \cup Y}$ so that $\{ 1, (a,b) \} \subseteq G_{X,Y}$.  Then
    $$
    \t \lhd \overline{(a,b) \cdot \t} \lhd \overline{\tau \cdot \t}
    $$
    for all $\tau \in G_{X,Y} \setminus \{ 1, (a,b) \}$, and
    $$
    e_{\t} = -e_{(a,b) \cdot \t} - \sum_{\tau \in G_{X,Y} \setminus \{1,(a,b)\}} e_{\tau \cdot \t}.
    $$
\item 
The polytabloid $e_{\t}$ lies in the $\ZZ$-span of $\{ e_{\s} \mid \s \in \STab(\lambda),\ \s \rhd \t\}$.
\item Suppose that, for some $k,l \in \ZZ^+$, we have:
\begin{enumerate}
\item[\normalfont {(I)}] $\t {\downarrow_{k}}$ and the subtableau of $\t$ consisting of the first $l$ rows are both column standard;
\item[\normalfont {(II)}] $\Shape(\t {\downarrow_k})$ has either $l$ or $l+1$ rows, and its $l$-th row is at least as long as the $(l+1)$-th row of $\lambda$.
\end{enumerate}
Then $e_{\t}$ lies in the $\ZZ$-spanned of $$\{ e_{\s} \mid \s \in \STab(\lambda),\ \s {\downarrow_k} = \t {\downarrow_k},\, \text{the first $l$ rows of $\s$ are the same as those of $\t$} \}.$$
\end{enumerate}
\end{prop}

We provide a proof below for the assertion about the dominance order in part (1); the remaining are direct consequences of \eqref{E:row-equiv} and \eqref{E:Garnir-2}.

\begin{proof}
Recall the following fact from \cite[3.7 Lemma]{Mathasbook99}: if $c$ lies in a higher row than $d$ in $\t$ and $c > d$, then $\overline{(c,d) \cdot \t} \rhd \t$, which immediately yields $\overline{(a,b) \cdot \t} \rhd \t$.
    For $\overline{\tau \cdot \t} \rhd \overline{ (a,b) \cdot \t}$, we prove by induction on the number $m$ of elements in $X$ that lie in the $(i+1)$-th row of $\tau \cdot \t$ (which is also the number of elements in $Y$ that lie in the $i$-th row of $\tau \cdot \t$).
    If $m = 1$, let $a' \in X$ be in the $(i+1)$-th row of $\tau \cdot \t$ and let $b' \in Y$ be in the $i$-th row of $\tau \cdot \t$. Since $\tau \ne (a,b)$, we have $b' < b$ or $a' > a$, so that
    $$
    \overline{\tau \cdot \t} \unrhd \overline{(a,a') \cdot (\tau \cdot \t)} \unrhd \overline{(b,b')(a,a') \cdot (\tau \cdot \t)} = \overline{(a,b) \cdot \t}$$
    (where $(a,a')$ and $(b,b')$ are to be read as $1$ if $a=a'$ and $b=b'$ respectively), with at least one of the inequalities being strict.
    For $m>1$, take $a' \in X$ and $b' \in Y$ in the $(i+1)$-th and $i$-th rows of $\tau \cdot \t$ respectively, and let $\tau' \in G_{X,Y}$ such that $\tau' \cdot \t$ is row equivalent to $(a',b') \cdot (\tau \cdot \t)$ (two tableaux $\s$ and $\s'$ are row equivalent if $\s' = \rho \cdot \s$ for some $\rho \in R_{\s}$).  Then there are $m-1$ elements in $X$ that lie in the $(i+1)$-th row of $\tau' \cdot \t$, and so
    $$
    \overline{\tau \cdot \t} \rhd \overline{(a',b') \cdot (\tau \cdot \t)} = \overline{\tau' \cdot \t} \unrhd (a,b) \cdot \t,
    $$
    where the last inequality follows from induction.
%
%
%
%
\end{proof}

Given a commutative ring $\mathcal{O}$ with $1$, define $S^\mathcal{O}_{\lambda} := \mathcal{O} \otimes_{\ZZ} S_{\lambda}^{\ZZ}$.  All the above statements about $S^{\ZZ}_{\lambda}$ behave well under base change, so that analogous statements hold when $\ZZ$ is replaced by $\mathcal{O}$.
In this paper, we are most concerned when $\mathcal{O} = \QQ$, in which case the set $\{ S^{\QQ}_{\lambda} \mid \lambda \vdash n \}$ is a complete set of pairwise non-isomorphic irreducible $\QQ\sym{n}$-modules.

We have the following result on the coefficients of an element of $S^{\mathcal{O}}_{\lambda}$ when expressed in terms of its standard basis:

\begin{prop} \label{prop:semistandard}
Let $u \in S^{\mathcal{O}}_{\lambda}$, say $u= \sum_{\t \in \STab(\lambda)} a_{\t} e_{\t}$ where $a_{\t} \in \mathcal{O}$ for all $\t \in \STab(\lambda)$.  Let $I = \{ i \in [1,n-1] \mid \bs_i u = u \}$, and let $W_I = \left< \bs_i \mid i \in I \right>$.
\begin{enumerate}
  \item If $\t, \t' \in \STab(\lambda)$ such that $\sigma \cdot \t = \t'$ for some $\sigma \in W_I$, then $a_\t=a_{\t'}$.
  \item If $\t \in \STab(\lambda)$ and there exists some $(i,j) \in W_I$ such that $i$ and $j$ lie in the same column of $\t$, then $a_\t=0$.
\end{enumerate}
\end{prop}

\begin{proof}
First, fix $i \in I$, and let
\begin{align*}
S_1 &= \{ \t \in \STab(\lambda)\mid  \text{$i$ and $i+1$ lie in the same row of $\t$} \}; \\
S_2 &= \{ \t \in \STab(\lambda)\mid  \text{$i$ and $i+1$ lie in the same column of $\t$} \}; \\
S_3 &= \STab(\lambda) \setminus (S_1 \cup S_2).
\end{align*}
Then we have:

\begin{enumerate}
\item[(i)] $\bs_ie_{\t} = e_{\t}$ for all $\t \in S_1$;
\item[(ii)] $\bs_i e_{\t} = - e_{\t} + \sum_{\s \in \STab(\lambda): \s \rhd \t} c_{\s} e_{\s}$ for all $\t \in S_2$ by Proposition \ref{prop:Garnir}(1) and (2);
\item[(iii)] $\bs_i \cdot \t \in S_3$ (and so $\bs_i e_{\t} = e_{\bs_i \cdot \t}$) for all $\t \in S_3$.
\end{enumerate}

For each $j \in [1,3]$, let $V_j$ be the $\mathcal{O}$-span of the $e_{\t}$'s with $\t \in S_j$.
Then $S^{\mathcal{O}}_{\lambda} = V_1 \oplus V_2 \oplus V_3$. Let $u = u_1 + u_2 + u_3$, where $u_j\in V_j$ for all $j$.
From the above description of $\bs_ie_{\t}$, we have $\bs_i u_1 = u_1 \in V_1$ and $\bs_i u_3 \in V_3$.
Furthermore, if $u_2 \ne 0$, let $\t \in S_2$ be minimal (with respect to $\unrhd$) such that $a_{\t} \ne 0$; then the coefficient of $e_{\t}$ in $\bs_iu_2$ is $-a_{\t}$ by (ii) above, contradicting
$$u_1 + u_2+ u_3 = u = \bs_i u = \bs_i(u_1 + u_2 + u_3) = u_1 + \bs_iu_2 + \bs_iu_3.$$
Thus, $u_2 = 0$, i.e.\ $a_{\t} =0$ for all $\t \in \STab(\lambda)$ with $i$ and $i+1$ in the same column of $\t$, and so $\bs_i u_3 = u_3$, which yields $a_{\t} = a_{\bs_i \cdot \t}$ for all $\t \in S_3 = \{ \s \in \STab(\lambda) \mid \bs_i \cdot \s \in \STab(\lambda) \}$.

Now, if $\t, \sigma \cdot \t \in \STab(\lambda)$ where $\sigma \in W_I$,
then $\bs_{i_1} \cdot \t,(\bs_{i_2} \bs_{i_1}) \cdot \t, \dotsc, (\bs_{i_{\ell}} \dotsm \bs_{i_1}) \cdot \t \in \STab(\lambda)$ for some reduced expression $\sigma = \bs_{i_{\ell}}\dotsb\bs_{i_{1}}$ by Lemma \ref{lem:std}(1).
Since $W_I$ is a parabolic subgroup, we have $i_j \in I$ for all $j$,
so that
$$a_{\t} = a_{\bs_{i_1} \cdot \t} = a_{(\bs_{i_2} \bs_{i_1}) \cdot \t} = \dotsb = a_{(\bs_{i_{\ell}} \dotsm \bs_{i_1}) \cdot \t} = a_{\sigma \cdot \t}$$
by the paragraph above, proving part (1).

For part (2), if $\t \in \STab(\lambda)$ has $i$ and $j$ ($i< j$) lying in the same column with $(i,j) \in W_I$, then $[i,j-1] \subseteq I$ since $(i,j) = \bs_{j-1}\dotsm\bs_{i+1}\bs_i\bs_{i+1}\dotsm \bs_{j-1}$ is a reduced expression and $W_I$ is a parabolic subgroup.
As the entry $j'$ in $\t$ directly below $i$ satisfies $i < j' \leq j$, we therefore have $(i,j') \in W_I$.
Thus, by replacing $j$ with $j'$ if necessary, we may assume that $i$ and $j$ lie in adjacent rows.  Consequently, there exists $w \in \sym{[i,j]} \subseteq W_I$
such that $w \cdot \t \in \STab(\lambda)$ and $w(j) - w(i) = 1$ by Lemma \ref{lem:std}(2), so that $a_{\t} = a_{w \cdot \t} = 0$ by part (1) and the first paragraph respectively, since $w(i) \in [i,j-1] \subseteq I$.
\end{proof}

\subsection{Young's seminormal basis} \label{subsec:seminormal-basis}

Young's seminormal basis is first defined by Murphy in \cite{Murphy92} for the Iwahori-Hecke algebra of the symmetric group.  This induces Young's seminormal basis for the Specht modules of this algebra (see, for example, \cite[3.33]{Mathasbook99}) which satisfies a recurrence relation (see \cite[3.36 Theorem]{Mathasbook99}).  As we only need this recurrence relation and not the precise definition of Murphy (or Mathas), we define Young's seminormal basis (at the limit $q\to 1$) by this characterising relation.

\begin{prop} \label{prop:seminormal}
Let $\lambda \vdash n$.  For each $\s \in \STab(\lambda)$, there exists a unique $f_\s \in S_{\lambda}^{\QQ}$ such that the set $\{ f_{\s} \in S_{\lambda}^{\QQ} \mid \s \in \STab(\lambda) \}$ satisfies the following:
\begin{enumerate}
\item[\rm (i)] $f_{\IT^{\lambda}} = e_{\IT^{\lambda}}$;
\item[\rm (ii)] $f_{\bs_i \cdot \s} = -\tfrac{1}{r_{\s,i}}f_{\s} + \bs_i f_{\s}$ for any $i\in [1,n-1]$ and $\s,\, \bs_i \cdot \s \in \STab(\lambda)$ with $\bs_i \cdot \s \lhd \s$, where $r_{\s,i} = \res_{\s}(i+1) - \res_{\s}(i)$.
\end{enumerate}
Furthermore, $\{ f_\s \mid \s \in \STab(\lambda)\}$ is a basis for $S_{\lambda}^{\QQ}$, called {\em Young's seminormal basis}, and
$$
\bs_i f_{\s} =
\begin{cases}
   r_{\s,i} f_{\s} &\text{if $r_{\s,i}= \pm 1$}; \\
   \tfrac{1}{r_{\s,i}} f_{\s} + f_{\bs_i \cdot \s}, &\text{if $r_{\s,i} \leq -2$}; \\
      \tfrac{1}{r_{\s,i}} f_{\s} + (1-\frac{1}{r_{\s,i}^2}) f_{\bs_i \cdot \s}, &\text{if $r_{\s,i} \geq 2$}.
\end{cases}
$$
\end{prop}

\begin{proof}
Clearly, if there exists $\{ f_{\s} \in S_{\lambda}^{\QQ} \mid \s \in \STab(\lambda) \}$ satisfying the required conditions, then the $f_{\s}$'s can be easily seen to be unique by induction on $\s$ with the order $\unrhd$  by Lemma \ref{lem:std}(1), since the conditions prescribe $f_{\IT^{\lambda}}$, and $f_{\bs_i \cdot \s}$ in terms of $f_{\s}$ whenever $\bs_i \cdot \s \in \STab(\lambda)$ with $\bs_i \cdot \s \lhd \s$.

For existence, note that $\{ f_{\s} \mid \s \in \STab(\lambda)\}$ as defined in \cite[Definition 2.2(2)]{FLT} is a basis for $\QQ\sym{n}f_{\IT^{\lambda}}$ \cite[Theorem 2.3(6)]{FLT}, and that there is a $\QQ\sym{n}$-isomorphism $\phi : \QQ\sym{n}f_{\IT^{\lambda}} \to S^{\QQ}_{\lambda}$ sending $f_{\IT^{\lambda}}$ to $e_{\IT^{\lambda}}$. Since these $f_{\s}$'s satisfy $f_{\bs_i \cdot \s} = -\tfrac{1}{r_{\s,i}}f_{\s} + \bs_i f_{\s}$ whenever $\s, \bs_i \cdot \s \in \STab(\lambda)$ with $\bs_i \cdot \s \lhd \s$ by \cite[Theorem 2.3(4)]{FLT}, they are indeed the ones stipulated by the proposition once we identify them with their images under $\phi$.

The remaining assertion about $\bs_i f_{\s}$ also follows from \cite[Theorem 2.3(4)]{FLT}.
\end{proof}

\begin{defn} \label{defn:D}
Let $\lambda \vdash n$.
\begin{enumerate}
\item Denote the transition matrix from the standard basis to Young's seminormal basis of $S^{\QQ}_{\lambda}$ by $(q_{\uu, \vv})_{\uu, \vv \in \STab(\lambda)}$ (thus $f_{\s} = \sum_{\vv \in \STab(\lambda)} q_{\s,\vv}e_{\vv}$ for all $\s \in \STab(\lambda)$).
\item Define $D: \STab(\lambda) \to \QQ\sym{n}$ by $\DD{\s} = \sum_{\vv \in \STab(\lambda)} q_{\s,\vv} d(\vv)$ for all $\s \in \STab(\lambda)$. (Recall that $d(\vv)$ is the unique element in $\sym{n}$ such that $d(\vv) \cdot \IT^{\lambda} = \vv$.)
\item For each $\s \in \STab(\lambda)$, the {\em denominator} $\dd_{\s}$ of $f_{\s}$ is
the smallest positive integer $k$ such that $kf_\s$ lies in the $\ZZ$-span of the standard basis of $S^{\QQ}_{\lambda}$.
\end{enumerate}
\end{defn}

\begin{lem} \label{lem:f_in_terms_of_e}
Let $\lambda \vdash n$, and $\s,\t \in \STab(\lambda)$.  We have:
\begin{enumerate}
\item $q_{\s,\s} =1$, and $q_{\s,\t} =0$ unless $\t \unrhd \s$;
\item $f_{\s} = \DD{\s}\, e_{\t^{\lambda}} = \DD{\s} f_{\t^{\lambda}}$, and $\DD{\s} \in \QQ W$ for any parabolic subgroup $W$ of $\sym{n}$ with $d(\s) \in W$;
\item $\dd_{\s} = \operatorname{lcm}\{ \text{denominator of } q_{\s,\vv} \mid \vv\in \STab(\lambda) \} = \min\{ k \in \mathbb{Z}^+ \mid k\DD{\s} \in \ZZ\sym{n}\}$.
\end{enumerate}
\end{lem}

\begin{proof}
Part (1) is \cite[Proposition 2.5(4)]{FLT}, while the first assertion of part (2) is clear from the definition of $\DD{\s}$.
By (1), $\DD{\s} = d(\s) + \sum_{\vv \rhd \s} q_{\s,\vv} d(\vv)$.  Note that $\vv \rhd \s$ implies that $d(\vv)$ has a reduced expression which is a subexpression of a reduced expression of $d(\s)$ \cite[3.7 Lemma]{Mathasbook99}, so that $d(\vv) \in W$ when $W$ is a parabolic subgroup of $\sym{n}$ with $d(\s) \in W$.
Thus $\DD{\s} \in \QQ W$. Part (3) follows from parts (1) and
(2) immediately.
\end{proof}

The following result, which we require in this paper, is a generalisation of \cite[Proposition 2.5(2)]{FLT}.

\begin{prop}	\label{prop:same-relative-positions}
Let $I \subseteq [1,n-1]$, and $W_I = \left< \bs_i \mid i \in I \right>$.
Let $\s \in \STab(\lambda)$ for some $\lambda \vdash n$, and define $\Gamma_{I,\s} =\{\tau \in W_I \mid \tau \cdot \s \in \STab(\lambda)\}$.  Let $A \in \QQ W_I$.
\begin{enumerate}
\item
Then $A f_{\s}$ lies in the $\QQ$-span of $\{ f_{\tau \cdot \s} \mid \tau \in \Gamma_{I,s} \}$.
\item  If $\t\in \STab(\nu)$ (for another partition $\nu$) and $z  \in \ZZ$ satisfy $i+z \in [1,|\nu|-1]$ and $\res_{\s}(i+1) - \res_{\s}(i) = \res_{\t}(i+z+1) - \res_{\t}(i+z)$ for all $i \in I$, then
$$
A^{+z} f_{\t} = \sum_{\tau\in \Gamma_{I,\s}} a_{\tau} f_{\tau^{+z} \cdot \t}
$$
when $A f_{\s} = \sum_{\tau \in \Gamma_{I,\s}} a_{\tau} f_{\tau \cdot \s}$ (by (1), with $a_{\tau} \in \QQ$ for all $\tau \in \Gamma_{I,s}$).
\end{enumerate}
\end{prop}

\begin{proof}
We prove both parts together. It suffices to show these statements for $A = \sigma \in W_I$, and for this, we prove by induction on $\ell = \ell(\sigma)$, with $\ell = 0$ being trivial.  For $\ell>0$, let $\sigma = \bs_i \rho$, with $\ell(\rho) = \ell -1$ and $i \in I$.  By induction,
for every $\tau\in\Gamma_{I,\s}$, there exists $a'_{\tau} \in \QQ$ such that
$$
\rho f_{\s} = \sum_{\tau\in\Gamma_{I,\s}} a'_{\tau} f_{\tau \cdot \s}   \qquad \text{and} \qquad
\rho^{+z} f_{\t}  =  \sum_{\tau\in\Gamma_{I,\s}} a'_{\tau} f_{\tau^{+z} \cdot \t}.
$$
Thus,
$$\sigma f_{\s} = \sum_{\tau\in\Gamma_{I,\s}} a'_{\tau} (\bs_i f_{\tau \cdot \s})   \qquad \text{and} \qquad
\sigma^{+z} f_{\t}  =  \sum_{\tau\in\Gamma_{I,\s}} a'_{\tau} (\bs_i^{+z} f_{\tau^{+z} \cdot \t}).$$
For each $\tau\in\Gamma_{I,\s}$, we have $\bs_i f_{\tau\cdot \s} = b_{\tau} f_{\tau \cdot \s} + c_{\tau} f_{(\bs_i\tau) \cdot \s}$ by Proposition \ref{prop:seminormal}, where $b_\tau, c_{\tau} \in \QQ$ are completely determined by $r_{\tau\cdot \s,i} = \res_{\tau \cdot \s}(i+1) - \res_{\tau\cdot \s}(i)$, with $c_{\tau} = 0$ when $r_{\tau\cdot \s,i} \in \{\pm1\}$, or equivalently, when $(\bs_i\tau) \cdot \s \notin \STab(\lambda)$. Part (1) thus follows.

For part (2), note first that for $1 \leq a_1 < a_2 \leq n$ with $\sym{[a_1,a_2]} \subseteq W_I$, we have
\begin{align*}
\res_{\s}(a_2) - \res_{\s}(a_1)
= \sum_{i=a_1}^{a_2-1} \res_{\s}(i+1)-\res_{\s}(i)
&= \sum_{i=a_1}^{a_2-1} \res_{\t}(i+z+1)-\res_{\t}(i+z) \\
&= \res_{\t}(a_2+z) - \res_{\t}(a_1+z).
\end{align*}
In particular, for $\rho \in W_I$ and $i \in I$, since the largest integer interval $J$ containing $i$ such that $\sym{J} \subseteq W_I$ also contains $i+1$, and $\rho(J) = J$, we have
$$
\res_{\s}(\rho(i+1)) - \res_{\s}(\rho(i))
= \res_{\t}(\rho(i+1)+z)-\res_{\s}(\rho(i)+z).
$$
Next, note also that $\res_{\sigma \cdot \uu}(j) = \res_{\uu}(\sigma^{-1}(j))$ for all $\sigma \in \sym{n}$, $\uu \in \Tab(\lambda)$ and $j\in [1,n]$.
Thus, for $\tau \in W_I$ and $i \in I$, we have
\begin{align*}
r_{\tau \cdot \s,i} = \res_{\tau \cdot \s}(i+1) - \res_{\tau\cdot \s}(i) &= \res_{\s}(\tau^{-1}(i+1)) - \res_{\s}(\tau^{-1}(i)) \\
&= \res_{\t}(\tau^{-1}(i+1) + z) - \res_{\t}(\tau^{-1}(i) + z) \\
&= \res_{\t} ((\tau^{+z})^{-1}(i+1+z)) - \res_{\t} ((\tau^{+z})^{-1}(i+z)) \\
&= \res_{\tau^{+z}\cdot \t}(i+z+1) - \res_{\tau^{+z}\cdot \t}(i+z) = r_{\tau^{+z} \cdot \t, i+z},
\end{align*}
so that if $\bs_i f_{\tau\cdot \s} = b_{\tau} f_{\tau \cdot \s} + c_{\tau} f_{(\bs_i\tau) \cdot \s}$ then
$$\bs_i^{+z} f_{\tau^{+z} \cdot \t} = \bs_{i+z} f_{\tau^{+z} \cdot \t} = b_{\tau} f_{\tau^{+z} \cdot \t} + c_{\tau} f_{(\bs_{i+z}\tau^{+z}) \cdot \t} = b_{\tau} f_{\tau^{+z} \cdot \t} + c_{\tau} f_{(\bs_i\tau)^{+z} \cdot \t}$$
by Proposition \ref{prop:seminormal}.  Part (2) thus follows.
\end{proof}

\begin{cor} \label{cor:same-relative-positions}
Let $\lambda = (\lambda_1,\dotsc, \lambda_r) $ and $\nu = (\nu_1,\dotsc, \nu_t)$ be partitions with $[\lambda] \subseteq [\nu]$.
\begin{enumerate}[itemsep=3pt]
\item Let $\s,\t \in \STab(\lambda)$.  If $f_{\t} = A f_{\s}$ for some $A \in \QQ\sym{|\lambda|}$, then $f_{\upt{\t}{\nu}} = A f_{\upt{\s}{\nu}}$.

\item For any $m \in [1,\nu_1-\lambda_1]$, we have
    $$
    f_{\upI{\lambda}{\nu}} = \DD{\upa{\lambda}{m}} f_{\upI{\lambda+(m)}{\nu}} =
    \DD{\upa{\lambda}{m}} \DD{\upI{\lambda+(m)}{\nu}} e_{\IT^{\nu}}.$$

\item For $m \in \ZZ^+$ and $i \in [2,r-1]$, we have
\begin{align*}
f_{\upa{\lambda}{m}}
&= \DD{\upa{\lambda^{(i)}}{m}}^{+(|\lambda| - |\lambda^{(i)}|)} f_{\upI{\lambda^{\leqslant i}}{\lambda+(m)}} \\
&= \DD{\upa{\lambda^{(i)}}{m}}^{+(|\lambda| - |\lambda^{(i)}|)} \DD{\upa{\lambda^{\leqslant i}}{m}} e_{\IT^{\lambda+(m)}},
\end{align*}
where $\lambda^{(i)} = (\lambda_1+i-1,\lambda_{i+1},\dotsc, \lambda_r)$ and $\lambda^{\leqslant i} = (\lambda_1,\dotsc, \lambda_i)$.
\end{enumerate}
\end{cor}

\begin{proof}
Part (1) follows from Proposition \ref{prop:same-relative-positions}(2) with $z=0$, and part (2) follows from part (1) (with $\t = \upa{\lambda}{m}$, $\s = \IT^{\lambda+(m)}$ and $A = \DD{\upa{\lambda}{m}}$; note that $\upt{\upa{\lambda}{m}}{\nu} = \upI{\lambda}{\nu}$).

For part (3), let $\s = \IT^{\lambda^{(i)}+(m)}$, $\t = \upI{\lambda^{\leqslant i}}{\lambda+(m)}$, and $z = |\lambda| - |\lambda^{(i)}|$.
Then $\res_{\s}(j) = \res_{\t}(j+z)+i-1$ for all $j \in [\lambda_1+i,|\lambda^{(i)}|+m]$.
Since
$$
f_{d(\upa{\lambda^{(i)}}{m}) \cdot \s} =
f_{\upa{\lambda^{(i)}}{m}} = \DD{\upa{\lambda^{(i)}}{m}} f_{\s},
$$
and $d(\upa{\lambda^{(i)}}{m}) \in \sym{[\lambda_1+i,|\lambda^{(i)}|+m]}$ so that $\DD{\upa{\lambda^{(i)}}{m}} \in \QQ\sym{[\lambda_1+i,|\lambda^{(i)}|+m]}$
by Lemma \ref{lem:f_in_terms_of_e}(2), we have
$$
f_{\upa{\lambda}{m}}=f_{d(\upa{\lambda^{(i)}}{m})^{+z} \cdot \t} =\DD{\upa{\lambda^{(i)}}{m}}^{+z} f_{\t}
$$
by Proposition \ref{prop:same-relative-positions}(2).
Since $f_{\upa{\lambda^{\leqslant i}}{m}} = \DD{\upa{\lambda^{\leqslant i}}{m}} f_{\IT^{\lambda^{\leqslant i} + (m)}}$ by Lemma \ref{lem:f_in_terms_of_e}(2), part (1) applies to yield
\begin{align*}
f_{\t} = f_{\upt{\upa{\lambda^{\leqslant i}}{m}}{\lambda+(m)}} &= \DD{\upa{\lambda^{\leqslant i}}{m}} f_{\upI{\lambda^{\leqslant i} + (m)}{\lambda+(m)}} \\
&= \DD{\upa{\lambda^{\leqslant i}}{m}} f_{\IT^{\lambda+(m)}} = \DD{\upa{\lambda^{\leqslant i}}{m}} e_{\IT^{\lambda+(m)}}.
\end{align*}
Thus,
$$
f_{\upa{\lambda}{m}}  = \DD{\upa{\lambda^{(i)}}{m}}^{+z} f_{\t} = \DD{\upa{\lambda^{(i)}}{m}}^{+z} \DD{\upa{\lambda^{\leqslant i}}{m}} e_{\IT^{\lambda+(m)}}.
$$
\end{proof}

This paper focuses on Young's seminormal basis vectors of the form $f_{\upI{\lambda}{\nu}}$, where $\lambda$ and $\nu$ are partitions with $[\lambda] \subseteq [\nu]$.
We end this section by making a useful observation about the coefficient of $e_{\s}$ in $f_{\upI{\lambda}{\nu}}$.

\begin{defn} \label{defn:colour}
Fix countably infinitely many colours $c_1,c_2,\dotsc$, and order them according to the natural order of their subscripts (i.e.\ $c_1< c_2 < \dotsb$).
Let $\lambda$  and $\nu$ be partitions such that $[\lambda] \subseteq [\nu]$.  Write $\nu_{\lambda}$ for the composition obtained by concatenating $\lambda$ and $\nu - \lambda$, i.e.\ $\nu_{\lambda} = (\lambda_1, \dotsc, \lambda_r, \nu_1-\lambda_1,\dotsc, \nu_r - \lambda_r, \nu_{r+1},\dotsc, \nu_t)$ if $\lambda = (\lambda_1,\dotsc, \lambda_r)$ and $\nu = (\nu_1,\dotsc, \nu_t)$.

\begin{enumerate}
\item For each $i \in [1,|\nu|]$, define the colour of $i$ to be $c_j$ if $i$ lies in the $j$-th row of $\IT^{\nu_{\lambda}}$.

  \item Two $\nu$-tableaux $\s$ and $\t$ are said to have the {\em same $\nu_{\lambda}$-colour type} if $\s(i,j)$ and $\t(i,j)$ are of the same colour for all $(i,j) \in [\nu]$.
  \item Given $\t \in \STab(\nu)$, we say that $\t$ is {\em colour-semistandard of type $\nu_{\lambda}$} if the colours of the integers appearing in $\t$ are strictly increasing down each column.
      The set of standard $\nu$-tableaux that are colour-semistandard of type $\nu_{\lambda}$ shall be denoted as $\SSTab(\lambda; \nu-\lambda)$.
\end{enumerate}
\end{defn}

\begin{rem} \hfill
\begin{enumerate}[itemsep=3pt]
\item For $\t \in \STab(\nu)$, let $\nu_{\lambda}(\t)$ be the $\nu$-tableau (of type $\nu_{\lambda}$) obtained from $\t$ by replacing each $i$ appearing in $\t$ by $r_i$ when $i$ appears in the $r_i$-th row of $\IT^{\nu_{\lambda}}$.  Then $\t$ is colour-semistandard of type $\nu_{\lambda}$ if and only if $\nu_{\lambda}(\t)$ is semistandard as a $\nu$-tableau of type $\nu_{\lambda}$ in the usual sense, i.e.\ having entries that are weakly increasing along each row, and strictly increasing down each column.

\item We shall often omit any mention of $\nu_{\lambda}$ when this is obvious from the context, and simply say `colour-semistandard' and `colour type'.

\item We write $\SSTab(\lambda;m)$ for $\SSTab(\lambda;(m))$ (when $\nu = \lambda + (m)$).
\end{enumerate}
\end{rem}

\begin{eg}\label{Eg: colour-ss}
Let $\lambda = (3,3,2,2)$ and $\nu=(4,3,2,2)$ so that $\nu_\lambda=(3,3,2,2,1)$.  The following is a complete list of representatives $\s\in \SSTab(\lambda;1)$ with distinct colour types:
\[\ytableausetup{notabloids}
\begin{matrix}
\s_0 \\[3pt]
\begin{ytableau}
\scriptstyle 1 & \scriptstyle 2 & \scriptstyle 3 & \scriptstyle \textcolor{teal}{11}  \\
\scriptstyle \textcolor{red}{4} & \scriptstyle \textcolor{red}{5} & \scriptstyle \textcolor{red}{6} \\
\scriptstyle \textcolor{blue}{7} & \scriptstyle \textcolor{blue}{8} \\
\scriptstyle \textcolor{violet}{9} & \scriptstyle \textcolor{violet}{10}
\end{ytableau}
\end{matrix}
\qquad
\begin{matrix}
\s_1 \\[3pt]
\begin{ytableau}
\scriptstyle 1 & \scriptstyle 2 & \scriptstyle 3 & \scriptstyle \textcolor{red}{6}  \\
\scriptstyle \textcolor{red}{4} & \scriptstyle \textcolor{red}{5} & \scriptstyle \textcolor{teal}{11} \\
\scriptstyle \textcolor{blue}{7} & \scriptstyle \textcolor{blue}{8} \\
\scriptstyle \textcolor{violet}{9} & \scriptstyle \textcolor{violet}{10}
\end{ytableau}
\end{matrix}
\qquad
\begin{matrix}
\s_2 \\[3pt]
\begin{ytableau}
\scriptstyle 1 & \scriptstyle 2 & \scriptstyle 3 & \scriptstyle \textcolor{violet}{10}  \\
\scriptstyle \textcolor{red}{4} & \scriptstyle \textcolor{red}{5} & \scriptstyle \textcolor{red}{6} \\
\scriptstyle \textcolor{blue}{7} & \scriptstyle \textcolor{blue}{8} \\
\scriptstyle \textcolor{violet}{9} & \scriptstyle \textcolor{teal}{11}
\end{ytableau}
\end{matrix}
\qquad
\begin{matrix}
\s_3 \\[3pt]
\begin{ytableau}
\scriptstyle 1 & \scriptstyle 2 & \scriptstyle 3 & \scriptstyle \textcolor{blue}{8}  \\
\scriptstyle \textcolor{red}{4} & \scriptstyle \textcolor{red}{5} & \scriptstyle \textcolor{red}{6} \\
\scriptstyle \textcolor{blue}{7} & \scriptstyle \textcolor{violet}{10} \\
\scriptstyle \textcolor{violet}{9} & \scriptstyle \textcolor{teal}{11}
\end{ytableau}
\end{matrix}
\qquad
\begin{matrix}
\s_4 \\[3pt]
\begin{ytableau}
\scriptstyle 1 & \scriptstyle 2 & \scriptstyle 3 & \scriptstyle \textcolor{red}{6}  \\
\scriptstyle \textcolor{red}{4} & \scriptstyle \textcolor{red}{5} & \scriptstyle \textcolor{violet}{10} \\
\scriptstyle \textcolor{blue}{7} & \scriptstyle \textcolor{blue}{8} \\
\scriptstyle \textcolor{violet}{9} & \scriptstyle \textcolor{teal}{11}
\end{ytableau}
\end{matrix}
\qquad
\begin{matrix}
\s_5 \\[3pt]
\begin{ytableau}
\scriptstyle 1 & \scriptstyle 2 & \scriptstyle 3 & \scriptstyle \textcolor{red}{6}  \\
\scriptstyle \textcolor{red}{4} & \scriptstyle \textcolor{red}{5} & \scriptstyle \textcolor{blue}{8} \\
\scriptstyle \textcolor{blue}{7} & \scriptstyle \textcolor{violet}{10} \\
\scriptstyle \textcolor{violet}{9} & \scriptstyle \textcolor{teal}{11}
\end{ytableau}
\end{matrix}
\]

The following standard tableaux on the other hand are not colour-semistandard:
\[
\begin{ytableau}
\scriptstyle 1 & \scriptstyle 2 & \scriptstyle 3 & \scriptstyle \textcolor{teal}{11}  \\
\scriptstyle \textcolor{red}{4} & \scriptstyle \textcolor{red}{5} & \scriptstyle \textcolor{blue}{7} \\
\scriptstyle \textcolor{red}{6} & \scriptstyle \textcolor{blue}{8} \\
\scriptstyle \textcolor{violet}{9} & \scriptstyle \textcolor{violet}{10}
\end{ytableau}
\qquad
\begin{ytableau}
\scriptstyle 1 & \scriptstyle 2 & \scriptstyle 3 & \scriptstyle \textcolor{violet}{9}  \\
\scriptstyle \textcolor{red}{4} & \scriptstyle \textcolor{blue}{7} & \scriptstyle \textcolor{blue}{8} \\
\scriptstyle \textcolor{red}{5} & \scriptstyle \textcolor{violet}{10} \\
\scriptstyle \textcolor{red}{6} & \scriptstyle \textcolor{teal}{11}
\end{ytableau}
\qquad
\begin{ytableau}
\scriptstyle 1 & \scriptstyle 2 & \scriptstyle 3 & \scriptstyle \textcolor{violet}{10}  \\
\scriptstyle \textcolor{red}{4} & \scriptstyle \textcolor{red}{5} & \scriptstyle \textcolor{red}{6} \\
\scriptstyle \textcolor{blue}{7} & \scriptstyle \textcolor{violet}{9} \\
\scriptstyle \textcolor{blue}{8} & \scriptstyle \textcolor{teal}{11}
\end{ytableau}
\]
\smallskip
\end{eg}

\begin{lem} \label{lem:easy-lem-for-semistd}
Let $\lambda = (\lambda_1,\dotsc, \lambda_r)$  and $\nu$ be partitions such that $[\lambda] \subseteq [\nu]$, and let $\s \in \STab(\nu)$ be colour-semistandard.
Then for each $i \in [1,r]$, any integer in the $i$-th row of $\s$ has colour $c_{j'}$ with $j' \geq i$.  Equivalently, all integers with colour $c_j$ for $j \in [1,i]$ appear in the first $i$ rows of $\s$.
\end{lem}

\begin{proof}
This is clear since the colours of the integers appearing in $\s$ are strictly increasing down each column.
\end{proof}

\begin{thm} \label{lem:semistandard}
Let $\lambda$  and $\nu$ be partitions such that $[\lambda] \subseteq [\nu]$.  For $\s,\t \in \STab(\nu)$, we have
\begin{enumerate}
\item $q_{\upI{\lambda}{\nu},\s} = q_{\upI{\lambda}{\nu},\t}$ if $\s$ and $\t$ have the same colour type;
\item $q_{\upI{\lambda}{\nu},\s} = 0$ if $\s$ is not colour-semistandard.
\end{enumerate}
\end{thm}

\begin{proof}
Since $f_{\upI{\lambda}{\nu}} = \sum_{\s \in \STab(\nu)} q_{\upI{\lambda}{\nu},\s} e_{\s}$, the lemma follows from Propositions \ref{prop:semistandard} and \ref{prop:seminormal}, as $\sym{\nu_{\lambda}}$ leaves the rows of $\upI{\lambda}{\nu}$ invariant, and $\s$ and $\t$ have the same colour type if and only if $\sigma \cdot \s = \t$ for some $\sigma \in \sym{\nu_{\lambda}}$.
\end{proof}

\section{Young's seminormal basis vector $f_{\upa{\lambda}{1}}$} \label{sec:mu=(1)}

In this section, we determine a closed formula of $f_{\upa{\lambda}{1}}$ when expressed in terms of the standard basis of $S^\QQ_{\lambda+(1)}$, where $\lambda$ is any partition.

Throughout this section, let $\lambda = (\lambda_1,\dotsc,\lambda_r) \vdash n$.  We first study the tableaux in $\SSTab(\lambda;1)$.
For each $\s \in \SSTab(\lambda;1)$, let
$$
Q(\s) = \{ i \in [2,r] \mid \text{the $i$-th row of $\s$ contains some integer with colour not equal to $c_i$} \}.
$$
Clearly, $Q(\s) = \varnothing$ if and only if $\s = \upa{\lambda}{1}$.

\begin{prop} \label{prop:semistd-colour-type}
If $\s \in \SSTab(\lambda;1)$ with $Q(\s) = \{ i_1 < i_2 < \dotsb  < i_s \} \ne \varnothing$, then the colour type of $\s$ is as follows:
\begin{enumerate}
\item its first row consists of all integers with colour $c_1$ together with one integer with colour $c_{i_1}$;
\item for $i \in [2,r] \setminus Q(\s)$, its $i$-th row consists only of all integers with colour $c_i$;
\item for $j \in [1,s-1]$, its $i_j$-th row consists of $\lambda_{i_j}-1$ integers with colour $c_{i_j}$, together with one integer with colour $c_{i_{j+1}}$;
\item its $i_s$-th row consists of $\lambda_{i_s}-1$ integers with colour $c_{i_s}$, together with one integer with colour $c_{r+1}$.
\end{enumerate}
\end{prop}

\begin{proof}
By definition of $Q(\s)$, for $i \notin Q(\s)$, the $i$-th row of $\s$ contains only integers with colour $c_i$, and hence all the integers with colour $c_i$, since the $i$-th row of $\s$ has $\lambda_i$ nodes and there are exactly $\lambda_i$ integers with colour $c_i$.

By Lemma \ref{lem:easy-lem-for-semistd}, the first row of $\s$ contains all the $\lambda_1$ integers with colour $c_1$, and another integer with colour $c_a$ say.  Since $Q(\s) \ne \emptyset$, we have $\s \ne \upa{\lambda}{1}$, so that $a \ne r+1$, as otherwise $\s = \upa{\lambda}{1}$ by Lemma \ref{lem:easy-lem-for-semistd}.
Thus $a \in [2,r]$, and since the remaining $\lambda_a-1$ integers with colour $c_a$ are insufficient to fill up the $a$-th row of $\s$ (which has $\lambda_a$ nodes), we have $a \in Q(\s)$.
Now for $i \in [2,a-1]$, the $i$-th row of $\s$ contains only all the integers with colour $c_i$ by Lemma \ref{lem:easy-lem-for-semistd}.  Consequently $a = \min(Q(\s)) = i_1$.

For $i_j \in Q(\s)$, we may assume that the first $i_j-1$ rows of $\s$ are as described, and so these rows contain exactly all the integers with colour $c_b$ for $b \in [1,i_j-1]$, and one other integer with colour $c_{i_j}$.
By Lemma \ref{lem:easy-lem-for-semistd}, the $i_j$-th row of $\s$ contains the remaining $\lambda_{i_j}-1$ integers with colour $c_{i_j}$, and another integer with colour $c_{b'}$ say, with $b' > i_j$,
and for all $i \in [i_j+1,b'-1]$, the $i$-th row of $\s$ contains only all the integers with colour $c_i$, and hence $i \notin Q(\s)$.
If $b' = r+1$, then this shows that $i_j = \max(Q(\s)) = i_s$.
On the other hand, if $b' \ne r+1$, then $b' \in Q(\s)$ since the remaining $\lambda_{b'}-1$ integers with colour $c_{b'}$ are insufficient to fill up the $b'$-th row of $\s$ (which has $\lambda_{b'}$ nodes).
Since $i \notin Q(\s)$ for all $i \in [i_j+1,b'-1]$, this yields $b'= i_{j+1}$ as desired, and our proof is complete.
\end{proof}

\begin{lem} \label{lem:Q}
Suppose that $\s\in\SSTab(\lambda;1)$. If $i \in Q(\s)$ and $\lambda_i = \lambda_{i+1}$, then $i+1 \in Q(\s)$.
\end{lem}

\begin{proof}
If $i \in Q(\s)$ and $\lambda_i = \lambda_{i+1}$, then the $i$-th row of $\s$ contains some integer with colour $c_j$, with $j > i$ by Proposition \ref{prop:semistd-colour-type}.
If $i+1 \notin Q(\s)$ then the $(i+1)$-th row of $\s$ contains all the integers with colour $c_{i+1}$ by Proposition \ref{prop:semistd-colour-type}, so that $j\ne i+1$ and hence $j > i+1$.
But then the integer with colour $c_j$ in the $i$-th row of $\s$ is now above an integer with colour $c_{i+1}$ in the $(i+1)$-th row, contradicting $\s$ being standard.
\end{proof}

Lemma \ref{lem:Q} suggests that we can reduce $Q(\s)$ to a subset $P(\s)$ from which we can reconstruct $Q(\s)$, as follows.
For each $a \in \ZZ^+$, let $Q(\s)_a = \{ i \in Q(\s) \mid \lambda_i = a \}$.  Then $Q(\s) = \bigcup_{a\in \ZZ^+} Q(\s)_a$ (disjoint union).  Define
$$
P(\s) := \{ \min (Q(\s)_a) \mid a \in \ZZ^+,\, Q(\s)_a \ne \varnothing \}.
$$
Then for each $i \in P(\s)$, $Q(\s)_{\lambda_i} = \{ j \in [i,r] \mid \lambda_j = \lambda_i \}$ by Lemma \ref{lem:Q}, and $Q(\s) = \bigcup_{i \in P(\s)} Q(\s)_{\lambda_i}$.

\begin{eg}\label{Eg: colour-ss2} Continuing with our running example, Example \ref{Eg: colour-ss}, we have the following table for the representatives $\s_i$ ($i \in [0,5]$) in $\SSTab((3,3,2,2);1)$ with distinct colour types.

\[
\begin{array}{c|cccccc}
\s &
\ytableausetup{notabloids}
\begin{matrix}
\s_0 \\[3pt]
\begin{ytableau}
\scriptstyle 1 & \scriptstyle 2 & \scriptstyle 3 & \scriptstyle \textcolor{teal}{11}  \\
\scriptstyle \textcolor{red}{4} & \scriptstyle \textcolor{red}{5} & \scriptstyle \textcolor{red}{6} \\
\scriptstyle \textcolor{blue}{7} & \scriptstyle \textcolor{blue}{8} \\
\scriptstyle \textcolor{violet}{9} & \scriptstyle \textcolor{violet}{10}
\end{ytableau}
\end{matrix}
&
\begin{matrix}
\s_1 \\[3pt]
\begin{ytableau}
\scriptstyle 1 & \scriptstyle 2 & \scriptstyle 3 & \scriptstyle \textcolor{red}{6}  \\
\scriptstyle \textcolor{red}{4} & \scriptstyle \textcolor{red}{5} & \scriptstyle \textcolor{teal}{11} \\
\scriptstyle \textcolor{blue}{7} & \scriptstyle \textcolor{blue}{8} \\
\scriptstyle \textcolor{violet}{9} & \scriptstyle \textcolor{violet}{10}
\end{ytableau}
\end{matrix}
&
\begin{matrix}
\s_2 \\[3pt]
\begin{ytableau}
\scriptstyle 1 & \scriptstyle 2 & \scriptstyle 3 & \scriptstyle \textcolor{violet}{10}  \\
\scriptstyle \textcolor{red}{4} & \scriptstyle \textcolor{red}{5} & \scriptstyle \textcolor{red}{6} \\
\scriptstyle \textcolor{blue}{7} & \scriptstyle \textcolor{blue}{8} \\
\scriptstyle \textcolor{violet}{9} & \scriptstyle \textcolor{teal}{11}
\end{ytableau}
\end{matrix}
&
\begin{matrix}
\s_3 \\[3pt]
\begin{ytableau}
\scriptstyle 1 & \scriptstyle 2 & \scriptstyle 3 & \scriptstyle \textcolor{blue}{8}  \\
\scriptstyle \textcolor{red}{4} & \scriptstyle \textcolor{red}{5} & \scriptstyle \textcolor{red}{6} \\
\scriptstyle \textcolor{blue}{7} & \scriptstyle \textcolor{violet}{10} \\
\scriptstyle \textcolor{violet}{9} & \scriptstyle \textcolor{teal}{11}
\end{ytableau}
\end{matrix}
&
\begin{matrix}
\s_4 \\[3pt]
\begin{ytableau}
\scriptstyle 1 & \scriptstyle 2 & \scriptstyle 3 & \scriptstyle \textcolor{red}{6}  \\
\scriptstyle \textcolor{red}{4} & \scriptstyle \textcolor{red}{5} & \scriptstyle \textcolor{violet}{10} \\
\scriptstyle \textcolor{blue}{7} & \scriptstyle \textcolor{blue}{8} \\
\scriptstyle \textcolor{violet}{9} & \scriptstyle \textcolor{teal}{11}
\end{ytableau}
\end{matrix}
&
\begin{matrix}
\s_5 \\[3pt]
\begin{ytableau}
\scriptstyle 1 & \scriptstyle 2 & \scriptstyle 3 & \scriptstyle \textcolor{red}{6}  \\
\scriptstyle \textcolor{red}{4} & \scriptstyle \textcolor{red}{5} & \scriptstyle \textcolor{blue}{8} \\
\scriptstyle \textcolor{blue}{7} & \scriptstyle \textcolor{violet}{10} \\
\scriptstyle \textcolor{violet}{9} & \scriptstyle \textcolor{teal}{11}
\end{ytableau}
\end{matrix}
\\ \\
Q(\s) & \varnothing & \{ 2\} & \{4\} & \{3,4\} & \{ 2,4 \} & \{ 2,3,4 \} \\ \\
P(\s) & \varnothing & \{ 2\} & \{4\} & \{3\}   & \{2,4\} & \{ 2,3 \}
\end{array}\]
\smallskip
\end{eg}

\begin{lem} \label{lem:same-colour-type}
Let $\s,\t \in \SSTab(\lambda;1)$. Then $\s$ and $\t$ have the same colour type if and only if $Q(\s)=Q(\t)$, if and only if $P(\s) = P(\t)$.
\end{lem}

\begin{proof}
By Proposition \ref{prop:semistd-colour-type}, $\s$ and $\t$ have the same colour type if and only if $Q(\s) = Q(\t)$.  The construction of $P(\s)$ from $Q(\s)$, and that of $Q(\s)$ from $P(\s)$, uses only information about $\lambda$ and not about $\s$.  Thus $Q(\s) = Q(\t)$ if and only if $P(\s) = P(\t)$.
\end{proof}

We can now state the main result of this section.

\begin{thm} \label{thm:mu=(1)}
Let $\lambda = (\lambda_1,\dotsc, \lambda_r) \vdash n$.  For each $\s \in \SSTab(\lambda;1)$, define
$$a_{\s} = (-1)^{|Q(\s)| - |P(\s)|} \prod_{i \in P(\s)} \frac{1}{\lambda_1 - \lambda_i + \max(Q(\s)_{\lambda_i})}.$$
Then
$$
f_{\upa{\lambda}{1}} = \sum_{\s \in \SSTab(\lambda;1)} a_{\s} e_{\s}.$$
In other words, $q_{\upa{\lambda}{1},\s} = a_{\s}$ if $\s \in \SSTab(\lambda;1)$ and $0$ otherwise.
\end{thm}

We remark that $Q(\s)$, and hence $P(\s)$ and $a_{\s}$, clearly depends only on the colour type of $\s$, so that Theorem \ref{thm:mu=(1)} is in agreement with Theorem \ref{lem:semistandard}.

\begin{eg} \label{E:mu=(1)}
Continuing with our running example (Examples \ref{Eg: colour-ss} and \ref{Eg: colour-ss2}), we have the corresponding $a_{\s}$ for each of the colour types of $\SSTab((3,3,2,2);1)$ as follows:
\[\begin{array}{c|cccccc}
\s&\s_0&\s_1&\s_2&\s_3 & \s_4 & \s_5 \\ \\
a_\s & 1 & \frac{1}{2} & \frac{1}{2} & -\frac{1}{2} & \frac{1}{10} & -\frac{1}{10}
\end{array}\]
By Theorem \ref{thm:mu=(1)}, $f_{\upa{(3,3,2,2)}{1}} = \sum_{\s \in \STab((4,3,2,2))} q_{\upa{(3,3,2,2)}{1}, \s} e_{\s}$ where $q_{\upa{(3,3,2,2)}{1}, \s} = a_{\s_i}$ if $\s$ has the same colour type as $\s_i$ ($i \in [0,5]$), and $0$ otherwise.
\end{eg}

\begin{proof}[Proof of Theorem \ref{thm:mu=(1)}]
Note first that if $r = 1$, then $\upa{\lambda}{1} = \IT^{\lambda+(1)}$, $\SSTab(\lambda;1) = \{ \IT^{\lambda+(1)} \}$, $P(\IT^{\lambda+(1)}) = \varnothing = Q(\IT^{\lambda+(1)})$ and so the theorem holds trivially.

We prove by induction on $n$, with the base case of $n=1$ already dealt with above, since $r=1$ in this case. For the inductive step, again we only need to consider the remaining case of $r>1$.
Let $\tilde{\lambda} = (\lambda_1,\dotsc,\lambda_{r-1}, \lambda_r -1)$.
Then $\tilde{\lambda} \vdash (n-1)$.
By induction, $f_{\upa{\tilde{\lambda}}{1}} = \sum_{\tilde{\s} \in \SSTab(\tilde{\lambda};1)} a_{\tilde{\s}} e_{\tilde{\s}}$, so that $\DD{\upa{\tilde{\lambda}}{1}} = \sum_{\tilde{\s} \in \SSTab(\tilde{\lambda};1)} a_{\tilde{\s}} d(\tilde{\s})$.
Thus, by Corollary \ref{cor:same-relative-positions}(2) (with $m=1$), we have
\begin{align*}
f_{\upI{\tilde{\lambda}}{\lambda+(1)}}
&= \DD{\upa{\tilde{\lambda}}{1}} e_{\IT^{\lambda + (1)}}
= \sum_{\tilde{\s} \in \SSTab(\tilde{\lambda};1)} a_{\tilde{\s}} d(\tilde{\s}) e_{\IT^{\lambda + (1)}} \\
&= \sum_{\tilde{\s} \in \SSTab(\tilde{\lambda};1)} a_{\tilde{\s}} e_{d(\tilde{\s}) \cdot \IT^{\lambda + (1)}}
= \sum_{\tilde{\s} \in \SSTab(\tilde{\lambda};1)} a_{\tilde{\s}} e_{\upt{\tilde{\s}}{\lambda+(1)}}.
\end{align*}
Thus,
\begin{align} \label{E:(1)-case}
f_{\upa{\lambda}{1}} = f_{\bs_{n} \cdot (\upI{\tilde{\lambda}}{\lambda+(1)})}
&= (\bs_{n} + \tfrac{1}{\lambda_1 - \lambda_r + r}) f_{\upI{\tilde{\lambda}}{\lambda+(1)}}
= (\bs_{n} + \tfrac{1}{\lambda_1 - \lambda_r + r}) \sum_{\tilde{\s} \in \SSTab(\tilde{\lambda};1)} a_{\tilde{\s}} e_{\upt{\tilde{\s}}{\lambda+(1)}}
\end{align}
by Proposition \ref{prop:seminormal}.

To continue, we split the tableaux in $\SSTab(\tilde{\lambda};1)$ into three types:
$\tilde{\s} \in \SSTab(\tilde{\lambda},1)$ is of type 1 if $n$ lies neither in the $r$-th row nor in the $\lambda_r$-th column, type $2$ if $n$ lies in the $r$-th row and type $3$ if $n$ lies in the $\lambda_r$-th column (in which case $n$ lies in the $(r-1)$-th row; this is only possible when $\lambda_{r-1} = \lambda_r$).
Similarly, we split the tableaux in $\SSTab(\lambda;1)$ into two types: $\s \in \SSTab(\lambda;1)$ is of type 1 if $n+1$ lies in the $r$-th row, and type $2$ if $n+1$ lies above the $r$-th row.
The subset of $\SSTab(\tilde{\lambda};1)$ (respectively $\SSTab(\lambda;1))$ consisting of tableaux of type $i$ shall be denoted $\SSTab(\tilde{\lambda};1)_i$ (respectively $\SSTab(\lambda;1)_i$).

We have a bijection $\SSTab(\tilde{\lambda};1) \to \SSTab(\lambda;1)_1$ defined by $\tilde{\s} \mapsto \upt{\tilde{\s}}{\lambda+(1)}$, with inverse $\s \mapsto \s{\downarrow_{n}}$.
We also have a bijection $\SSTab(\tilde{\lambda};1)_1 \to \SSTab(\lambda;1)_2$ defined by $\tilde{\s} \mapsto \bs_n \cdot (\upt{\tilde{\s}}{\lambda+(1)})$, with inverse $\s \mapsto (\bs_n \cdot \s) {\downarrow_n}$ (note that $n$ lies in the $r$-th row of $\s \in \SSTab(\lambda;1)_2$ by Proposition \ref{prop:semistd-colour-type}).
Also if $\tilde{\s} \in \SSTab(\tilde{\lambda};1)_2$ then $\overline{\bs_n \cdot (\upt{\tilde{\s}}{\lambda+(1)})} = \upt{\tilde{\s}}{\lambda+(1)}$.  We defer for the moment the discussion of $\bs_n \cdot (\upt{\tilde{\s}}{\lambda+(1)})$ for $\tilde{\s} \in \SSTab(\tilde{\lambda};1)_3$.

The following table gives a summary of the important information obtained from Proposition \ref{prop:semistd-colour-type} pertaining to $\upt{\tilde{\s}}{\lambda+(1)}$ and $\bs_n \cdot (\upt{\tilde{\s}}{\lambda+(1)})$ in terms of that pertaining to $\tilde{\s}$.
{\scriptsize
\begin{table}[h]
\begin{tabular}{cc|ccccc}
type of $\tilde{\s}$ & $\t$ &
$r \in Q(\tilde{\s})$? &$Q(\t)$ & $P(\t)$ & $a_{\t}$ & remarks\\[6pt]
\hline
1 & $\upt{\tilde{\s}}{\lambda+(1)}$ &
No & $Q(\tilde{\s}) \cup \{r\}$ & $P(\tilde{\s}) \cup \{r\}$ & $\tfrac{1}{\lambda_1 - \lambda_r + r}a_{\tilde{\s}}$ \\[8pt]
2 & $\upt{\tilde{\s}}{\lambda+(1)}$ &
Yes &$Q(\tilde{\s})$ & $\left\{\begin{matrix} P(\tilde{\s}) \hfill \\[4pt] P(\tilde{\s}) \setminus \{ r\} \end{matrix}\right.$ &
$\left\{\begin{matrix} \tfrac{\lambda_1-\lambda_r+r+1}{\lambda_1 - \lambda_r + r}a_{\tilde{\s}} \\[3pt] -\tfrac{(\lambda_1-\lambda_r+r+1)(\lambda_1-\lambda_r+r-1)}{\lambda_1-\lambda_r+r}a_{\tilde{\s}} \end{matrix}\right.$ &
$\left\{\begin{matrix} \text{if } Q(\tilde{\s})_{\lambda_r} = \varnothing \\[7pt] \text{if } Q(\tilde{\s})_{\lambda_r} \ne \varnothing \end{matrix}\right.$
\\[14pt]
3 & $\upt{\tilde{\s}}{\lambda+(1)}$ &
No & $Q(\tilde{\s}) \cup \{r\}$ & $P(\tilde{\s})$ & $-\tfrac{\lambda_1 - \lambda_r + r-1}{\lambda_1 - \lambda_r + r}a_{\tilde{\s}}$ \\[8pt]
1 & $\bs_n \cdot (\upt{\tilde{\s}}{\lambda+(1)})$ &
No & $Q(\tilde{\s})$ & $P(\tilde{\s})$ & $a_{\tilde{\s}}$ \\
\hline \\
\end{tabular}
\caption{Summary Table}
\end{table}
}

We provide a sketch to justify the second row of the table, which is the most difficult, and leave the other rows to the reader as easy exercises.
Suppose that $\tilde{\s} \in \SSTab(\tilde{\lambda};1)_2$.
The only integer that changes colour when going from $\tilde{\s}$ to $\upt{\tilde{\s}}{\lambda+(1)}$ is $n$, whose colour changes from $c_{r+1}$ to $c_r$.
Thus by Proposition \ref{prop:semistd-colour-type}, $r \in Q(\tilde{\s})_{\lambda_r-1} \cap Q(\upt{\tilde{\s}}{\lambda+(1)})_{\lambda_r}$, and $\varnothing = Q(\upt{\tilde{\s}}{\lambda+(1)})_{\lambda_r-1} = Q(\tilde{\s})_{\lambda_r-1} \setminus \{r\}$ while $Q(\upt{\tilde{\s}}{\lambda+(1)})_{\lambda_r} \setminus \{r\} = Q(\tilde{\s})_{\lambda_r}$, and $Q(\upt{\tilde{\s}}{\lambda+(1)})_{b} = Q(\tilde{\s})_{b}$ for all $b \ne \lambda_r,\lambda_r-1$.  This yields $Q(\upt{\tilde{\s}}{\lambda+(1)})$ and $P(\upt{\tilde{\s}}{\lambda+(1)})$.  For $a_{\upt{\tilde{\s}}{\lambda+(1)}}$, we need only to note in addition that $\max(Q(\upt{\tilde{\s}}{\lambda+(1)})_{\lambda_r}) = r$, while $\max(Q(\tilde{\s})_{\lambda_r}) = r-1$ if $Q(\tilde{\s})_{\lambda_r} \ne \varnothing$.

With Table 1, we can now proceed to finish the proof, dealing with the cases $\lambda_{r-1} \ne \lambda_{r}$ and $\lambda_{r-1} = \lambda_r$ separately.
\begin{description}
\item[Case 1. $\lambda_{r-1} \ne \lambda_r$]  In this case, we have $\SSTab(\tilde{\lambda};1) = \bigcup_{i=1}^2 \SSTab(\tilde{\lambda};1)_i$ and $Q(\tilde{\s})_{\lambda_r} = \varnothing$ for all $\tilde{\s} \in \SSTab(\tilde{\lambda},1)_2$.  Continuing from \eqref{E:(1)-case}, and using the above summary table whenever necessary, we get
    \begin{align*}
    &f_{\upa{\lambda}{1}} \\
    &= \sum_{\tilde{\s} \in \SSTab(\tilde{\lambda};1)_1} (a_{\tilde{\s}} e_{\bs_n \cdot (\upt{\tilde{\s}}{\lambda+(1)})} + \tfrac{a_{\tilde{\s}}}{\lambda_1-\lambda_r+r} e_{\upt{\tilde{\s}}{\lambda+(1)}} )
    + \sum_{\tilde{\s} \in \SSTab(\tilde{\lambda};1)_2} (1+ \tfrac{1}{\lambda_1 - \lambda_r + r}) a_{\tilde{\s}} e_{\upt{\tilde{\s}}{\lambda+(1)}} \\
    &= \sum_{\tilde{\s} \in \SSTab(\tilde{\lambda};1)_1} (a_{\bs_n \cdot (\upt{\tilde{\s}}{\lambda+(1)})} e_{\bs_n \cdot (\upt{\tilde{\s}}{\lambda+(1)})} + a_{\upt{\tilde{\s}}{\lambda+(1)}} e_{\upt{\tilde{\s}}{\lambda+(1)}} )
    + \sum_{\tilde{\s} \in \SSTab(\tilde{\lambda};1)_2} a_{\upt{\tilde{\s}}{\lambda+(1)}} e_{\upt{\tilde{\s}}{\lambda+(1)}} \\
    &= \sum_{\tilde{\s} \in \SSTab(\tilde{\lambda};1)_1} a_{\bs_n \cdot (\upt{\tilde{\s}}{\lambda+(1)})} e_{\bs_n \cdot (\upt{\tilde{\s}}{\lambda+(1)})} + \sum_{\tilde{\s} \SSTab(\tilde{\lambda};1)} a_{\upt{\tilde{\s}}{\lambda+(1)}} e_{\upt{\tilde{\s}}{\lambda+(1)}}\\
    &= \sum_{\s \in \SSTab(\lambda;1)_2} a_{\s} e_{\s} + \sum_{\s \in \SSTab(\lambda;1)_1} a_{\s} e_{\s} = \sum_{\s \in \SSTab(\lambda;1)} a_{\s} e_{\s}.
    \end{align*}

    \item[Case 2. $\lambda_r = \lambda_{r-1}$] We look at $S_i := \sum_{\tilde{\s} \in \SSTab(\tilde{\lambda};1)_i} (\bs_n + \tfrac{1}{\lambda_1-\lambda_r+r}) a_{\tilde{\s}} e_{\upt{\tilde{\s}}{\lambda+(1)}}$ for $i \in [1,3]$ separately.  We have, just as in Case 1,
    $$
    S_1 =
    \sum_{\tilde{\s} \in \SSTab(\tilde{\lambda};1)_1} (a_{\bs_n \cdot (\upt{\tilde{\s}}{\lambda+(1)})} e_{\bs_n \cdot (\upt{\tilde{\s}}{\lambda+(1)})} + a_{\upt{\tilde{\s}}{\lambda+(1)}} e_{\upt{\tilde{\s}}{\lambda+(1)}} ).
    $$
    Furthermore, with the help of Table 1, we have
    \begin{align*}
    S_2
    &
    =\sum_{\tilde{\s} \in \SSTab(\tilde{\lambda};1)_2} (1+ \tfrac{1}{\lambda_1 - \lambda_r + r}) a_{\tilde{\s}} e_{\upt{\tilde{\s}}{\lambda + (1)}} \\
    &= \left(\sum_{\substack{\tilde{\s} \in \SSTab(\tilde{\lambda};1)_2 \\ Q(\tilde{\s})_{\lambda_r} = \varnothing}} + \sum_{\substack{\tilde{\s} \in \SSTab(\tilde{\lambda};1)_2 \\ Q(\tilde{\s})_{\lambda_r} \ne \varnothing}}\right)   \tfrac{\lambda_1-\lambda_r+r+1}{\lambda_1-\lambda_r+r}\, a_{\tilde{\s}} e_{\upt{\tilde{\s}}{\lambda + (1)}} \\
    &= \sum_{\substack{\tilde{\s} \in \SSTab(\tilde{\lambda};1)_2 \\ Q(\tilde{\s})_{\lambda_r} = \varnothing}} a_{\upt{\tilde{\s}}{\lambda + (1)}} e_{\upt{\tilde{\s}}{\lambda + (1)}} +
    \sum_{\substack{\tilde{\s} \in \SSTab(\tilde{\lambda};1)_2 \\ Q(\tilde{\s})_{\lambda_r} \ne \varnothing}} -\tfrac{1}{\lambda_1-\lambda_r +r-1}\, a_{\upt{\tilde{\s}}{\lambda + (1)}} e_{\upt{\tilde{\s}}{\lambda + (1)}}.
    \end{align*}
    For $S_3$, observe first that for each $\tilde{\s} \in \SSTab(\tilde{\lambda};1)_3$, we have $\tilde{\s}(r-1,\lambda_r) = n$, and so the $r$-th row of $\tilde{\s}$ contains $[n-\lambda_r+1, n-1]$ by Proposition \ref{prop:semistd-colour-type}.
    For each $i \in [n-\lambda_r+1,n-1]$, let $\uu_{\tilde{\s},i} \in \RSTab(\lambda+(1))$ be the tableau obtained from $\upt{\tilde{\s}}{\lambda + (1)}$ by swapping $n$ in its $(r-1,\lambda_r)$-node with $i$ in its $r$-th row (and rearranging the $r$-th row so that it is increasing).
    Then $\uu_{\tilde{\s},i} \in \SSTab(\lambda;1)$ with the same colour type as $\upt{\tilde{\s}}{\lambda + (1)}$, so that $a_{\uu_{\tilde{\s},i}}  = a_{\upt{\tilde{\s}}{\lambda + (1)}}$ by Lemma \ref{lem:same-colour-type}.  Furthermore,
\begin{equation*}
\{ \uu_{\tilde{\s},i} \mid i \in [n-\lambda_1+1,n-1],\, \tilde{\s} \in \SSTab(\tilde{\lambda};1)_3 \} = \{ \upt{\tilde{\s}}{\lambda+(1)} \mid \tilde{\s} \in \SSTab(\tilde{\lambda};1)_2,\, Q(\tilde{\s})_{\lambda_r} \ne \varnothing \}.
\end{equation*}

Now, by Proposition \ref{prop:Garnir}(1),
$$
e_{\bs_n \cdot (\upt{\tilde{\s}}{\lambda+(1)})} = - e_{\upt{\tilde{\s}}{\lambda+(1)}} - \sum_{i= n-\lambda_r+1}^{n-1} e_{\uu_{\tilde{\s},i}}.
$$
    Thus,
    \begin{align*}
    S_3 &= \sum_{\tilde{\s} \in \SSTab(\tilde{\lambda};1)_3} (a_{\tilde{\s}} e_{\bs_n \cdot (\upt{\tilde{\s}}{\lambda + (1)})} + \tfrac{1}{\lambda_1-\lambda_r+r} a_{\tilde{\s}} e_{\upt{\tilde{\s}}{\lambda + (1)}}) \\
    &= \sum_{\tilde{\s} \in \SSTab(\tilde{\lambda};1)_3} (a_{\tilde{\s}}(-1 + \tfrac{1}{\lambda_1-\lambda_r+r}) e_{\upt{\tilde{\s}}{\lambda + (1)}} - \sum_{i=n-\lambda_r+1}^{n-1} a_{\tilde{\s}} e_{\uu_{\tilde{\s},i}} ) \\
    &= \sum_{\tilde{\s} \in \SSTab(\tilde{\lambda};1)_3} (a_{\upt{\tilde{\s}}{\lambda + (1)}} e_{\upt{\tilde{\s}}{\lambda+(1)}} +  \sum_{i=n-\lambda_r+1}^{n-1} \tfrac{\lambda_1 + \lambda_r + r}{\lambda_1 + \lambda_r + r-1} a_{\upt{\tilde{\s}}{\lambda + (1)}} e_{\uu_{\tilde{\s},i}}) \\
    &= \sum_{\tilde{\s} \in \SSTab(\tilde{\lambda};1)_3} (a_{\upt{\tilde{\s}}{\lambda + (1)}} e_{\upt{\tilde{\s}}{\lambda + (1)}} +  \sum_{i=n-\lambda_r+1}^{n-1} \tfrac{\lambda_1 + \lambda_r + r}{\lambda_1 + \lambda_r + r-1} a_{\uu_{\tilde{\s},i}} e_{\uu_{\tilde{\s},i}}) \\
    &= \sum_{\tilde{\s} \in \SSTab(\tilde{\lambda};1)_3} a_{\upt{\tilde{\s}}{\lambda + (1)}} e_{\upt{\tilde{\s}}{\lambda + (1)}} +  \sum_{\substack{\tilde{\s} \in \SSTab(\tilde{\lambda};1)_2 \\ Q(\tilde{\s})_{\lambda_r} \ne \varnothing}} \tfrac{\lambda_1 + \lambda_r + r}{\lambda_1 + \lambda_r + r-1} a_{\upt{\tilde{\s}}{\lambda + (1)}} e_{\upt{\tilde{\s}}{\lambda + (1)}}.
    \end{align*}
    Hence $f_{\upa{\lambda}{1}} = S_1 + S_2 + S_3 = \sum_{\s \in \SSTab(\lambda;1)} a_{\s} e_{\s}$ just like in Case 1, as desired.
    \end{description}
\end{proof}

\begin{cor} \label{cor:mu=(1)}
Let $A_0,\dotsc,A_t$ be all the removable nodes of $[\lambda+(1)]$, labelled from top to bottom (thus $A_0$ lies in the top row).  Then
$$
\dd_{\upa{\lambda}{1}} = \prod_{i=1}^t (\res(A_0) - \res(A_i)).
$$
\end{cor}

\begin{proof}
Firstly, $[\lambda+(1)]$ has a removable node on its top row, so $A_0 = (1,\lambda_1+1)$.
Let $\s \in \SSTab(\lambda;1)$.
For each $i \in P(\s)$, we have $\max(Q(\s)_{\lambda_i}) = \max\{ j \in [1,r] \mid \lambda_j = \lambda_i \}$, so that the node
$(\max(Q(\s)_{\lambda_i}), \lambda_i) = A_{n_i}$ for some $n_i \in [1,t]$.
Furthermore,
$$\lambda_1 - \lambda_i + \max(Q(\s)_{\lambda_i}) = \res(A_0) - \res(A_{n_i}).$$
Thus $|a_{\s}| = \prod_{i\in P(\s)} \frac{1}{\res(A_0) - \res(A_{n_i})}$.
This shows $\dd_{\upa{\lambda}{1}} \mid \prod_{i=1}^t (\res(A_0) - \res(A_{i}))$.
To see that we indeed have equality, let $\s$ be the tableau obtained from $\upa{\lambda}{1}$ by moving only its entries at the removable nodes, so that the entry at $A_i$ appears at $A_{i-1}$ for all $i \in [1,t]$, and the entry at $A_0$ (namely $n+1$) appears at $A_t$.
Then $\s \in \SSTab(\lambda;1)$ with $a_{\s} = \prod_{i=1}^t \frac{1}{\res(A_0) - \res(A_i)}$, and we are done.
\end{proof}

\begin{eg}
The removable nodes of $(4,3,2,2)$ are $A_0 = (1,4)$, $A_1 = (2,3)$ and $A_2 = (4,2)$.  By Corollary \ref{cor:mu=(1)},
\[
\dd_{\upa{(3,3,2,2)}{1}} = (\res(A_0) - \res(A_1))(\res(A_0) - \res(A_2)) = (3-1)(3-(-2)) = 10,
\]
agreeing of course with the expression of $f_{\upa{(3,3,2,2)}{1}}$ obtained in Example \ref{E:mu=(1)}.
 \end{eg}

\begin{rem}
We note that, with Corollary \ref{cor:mu=(1)} and \cite[Theorem 3.13]{FLT}, we can determine $\theta_{\lambda,(1)}$ and the explicit condition on $p$ for the splitting of the canonical morphism $\iota_{\lambda,(1)} : \Delta(\lambda+(1)) \to \Delta(\lambda) \otimes \Delta(1)$ over $\ZP$.
\end{rem}

\section{Young's seminormal basis vector $f_{\upa{(k,\ell^s)}{m}}$} \label{sec:kl^sm}

Throughout this section, let $k,\ell,m,s \in \ZZ^+$ with $k \geq \ell$, and let
$$\kl = (k,\ell^s).$$ We study Young's seminormal basis vector
$f_{\upa{(k,\ell^s)}{m}} = f_{\upa{\lambda}{m}}$.
By Theorem \ref{lem:semistandard}(2), we have
$$f_{\upa{\kl}{m}} = \sum_{\s \in \SSTab(\kl;m)} q_{\upa{\kl}{m},\s} e_{\s}.$$
We shall obtain closed formulae for $q_{\upa{(k,1^s)}{m},\s}$ and $q_{\upa{(k,\ell)}{m},\s}$, and some reduction results for the general case.

\begin{lem} \label{lem:sstd-description}
Let $\s \in \SSTab(\kl;m)$.  Then, for all $i \in [2,s+1]$, only the integers with colour $c_i$ and $c_{i+1}$ may appear in the $i$-th row of $\s$.

In particular, all integers in $[1,k]$ appear in the first row of $\s$.
\end{lem}

\begin{proof}
This is clear since $\s$ has $s+1$ rows while the integers appearing in $\s$ have $s+2$ colours which are strictly increasing down each column.
\end{proof}

\begin{cor} \label{cor:sum-of-weight}
Let $\s \in \SSTab(\kl;m)$, and for each $i \in [1,s]$, let $n_i(\s)$ be the number of integers having colour $c_{i+1}$ in the first row of $\s$.
Then:
\begin{enumerate}[itemsep = 3pt]
\item $\sum_{i=1}^s n_i(\s) \leq \min(\ell,m)$;
\item for each $i \in [2,s+1]$, the $i$-th row of $\s$ contains exactly $\ell - \sum_{r=1}^{i-1}n_r(\s)$ integers with colour $c_i$ and exactly $\sum_{r=1}^{i-1} n_r(\s)$ integers with colour $c_{i+1}$.
\end{enumerate}
\end{cor}

\begin{proof} \hfill
\begin{enumerate}
\item The first row of $\s$, which has length $k+m$, contains exactly $k$ integers with colour $c_1$ by Lemma \ref{lem:sstd-description}.
    Thus $m = \sum_{i=1}^{s+1} n_i(\s) \geq \sum_{i=1}^{s} n_i(\s)$.
    On the other hand, the second to $s$-th rows of $\s$ are filled with integers with colour $c_2,\dotsc,c_{s+1}$ only by Lemma \ref{lem:sstd-description}, so that there are exactly $\ell(s-1)$ such integers in these rows.
    This leaves at most $\ell$ integers with these colours in the first row of $\s$. Hence $\sum_{i=1}^s n_i(\s) \leq \ell$.

\item We prove by induction on $i$.  By Lemma \ref{lem:sstd-description}, the $\ell$ integers with colour $c_2$ lie in the first two rows of $\s$.  Since the first row of $\s$ contains exactly $n_1(\s)$ integers with colour $c_2$, the second row of $\s$ contains exactly $\ell - n_1(\s)$ integers with colour $c_2$, and consequently also contains exactly $n_1(\s)$ integers with colour $c_3$ by Lemma \ref{lem:sstd-description}.
    Thus, the statement holds for $i=2$.

    Assume that $i < s$ and that the $i$-th row of $\s$  contains exactly $\sum_{r=1}^{i-1} n_r(\s)$ integers with colour $c_{i+1}$ and exactly $\ell - \sum_{r=1}^{i-1} n_r(\s)$ integers with colour $c_i$.
    Since the $\ell$ integers with colour $c_{i+1}$ may only appear in the first, $i$-th and $(i+1)$-th row of $\s$ by Lemma \ref{lem:sstd-description}, and there are exactly $n_i(\s)$ integers with colour $c_{i+1}$ in the first row, the $(i+1)$-th row of $\s$ must contain exactly $\ell - \sum_{r=1}^{i} n_r(\s)$ integers with colour $c_{i+1}$, and hence exactly $\sum_{r=1}^{i} n_r(\s)$ integers with colour $c_{i+2}$, again by Lemma \ref{lem:sstd-description}.
\end{enumerate}
\end{proof}

\begin{defn}
Keeping the notation introduced in Corollary \ref{cor:sum-of-weight}, we call the sequence $\wt(\s):=(n_1(\s),\dotsc, n_s(\s))$ the {\em weight} of $\s \in \SSTab(\kl;m)$.
\end{defn}

\begin{eg} For \[\s=\raisebox{3mm}{\begin{ytableau}
\scriptstyle 1 & \scriptstyle 2 & \scriptstyle 3 & \scriptstyle 4 & \scriptstyle \textcolor{red}{7} & \scriptstyle \textcolor{blue}{10} \\
\scriptstyle \textcolor{red}{5} & \scriptstyle \textcolor{red}{6} & \scriptstyle \textcolor{blue}{8} \\
\scriptstyle \textcolor{blue}{9} & \scriptstyle \textcolor{violet}{11} & \scriptstyle \textcolor{violet}{12}
\end{ytableau}}\in\SSTab((4,3^2);2),\]
its weight $\wt(\s) = (1,1)$.
\end{eg}

Let
$$\BW = \BW^{\min(\ell,m)}_s = \{(w_1,\dotsc, w_s) \in (\ZZ_{\geq 0})^s \mid \sum_{i=1}^s w_i \leq \min(\ell,m) \}.$$
For each $\bw = (w_1,\dotsc, w_s) \in \BW$, let $\s_{\bw} = \s^{k,\ell,m}_{\bw}$ be the $(\klm)$-tableau obtained as follows:  starting from $\upa{\kl}{m}$,
working from $j = s$ down to $j=1$,
swap the rightmost $\sum_{i=1}^j w_i$ integers having colour $c_{j+1}$ in the $(j+1)$-th row 
with the leftmost $\sum_{i=1}^j w_i$ integers in the first row having colour $c_{j+2}$. 
A moment's thought, together with the worked example below, should convince the reader that $\s_{\bw} \in \SSTab(\kl;m)$ with weight $\bw$, and is in fact the smallest such (with respect to $\unrhd$).

\begin{eg}
Let $k = \ell = m = 3$, $s=2$ and $\bw=(2,1)$.  Then
$$
\upa{\kl}{m} = \raisebox{3mm}{\ytableausetup{notabloids}
\begin{ytableau}
\scriptstyle 1 & \scriptstyle 2 & \scriptstyle 3 & \scriptstyle \textcolor{violet}{10} & \scriptstyle \textcolor{violet}{11} & \scriptstyle \textcolor{violet}{12} \\
\scriptstyle \textcolor{red}{4} & \scriptstyle \textcolor{red}{5} & \scriptstyle \textcolor{red}{6} \\
\scriptstyle \textcolor{blue}{7} & \scriptstyle \textcolor{blue}{8} & \scriptstyle \textcolor{blue}{9}
\end{ytableau}}
\ \xrightarrow{j=2}\
\raisebox{3mm}{\begin{ytableau}
\scriptstyle 1 & \scriptstyle 2 & \scriptstyle 3 & \scriptstyle \textcolor{blue}{7} & \scriptstyle \textcolor{blue}{8} & \scriptstyle \textcolor{blue}{9} \\
\scriptstyle \textcolor{red}{4} & \scriptstyle \textcolor{red}{5} & \scriptstyle \textcolor{red}{6} \\
\scriptstyle \textcolor{violet}{10} & \scriptstyle \textcolor{violet}{11} & \scriptstyle \textcolor{violet}{12}
\end{ytableau}}
\ \xrightarrow{j=1}\
\s_{(2,1)} = \raisebox{3mm}{
\begin{ytableau}
\scriptstyle 1 & \scriptstyle 2 & \scriptstyle 3 & \scriptstyle \textcolor{red}{5} & \scriptstyle \textcolor{red}{6} & \scriptstyle \textcolor{blue}{9} \\
\scriptstyle \textcolor{red}{4} & \scriptstyle \textcolor{blue}{7} & \scriptstyle \textcolor{blue}{8} \\
\scriptstyle \textcolor{violet}{10} & \scriptstyle \textcolor{violet}{11} & \scriptstyle \textcolor{violet}{12}
\end{ytableau}} .
$$
\end{eg}

By Corollary \ref{cor:sum-of-weight}, we have:

\begin{cor} \label{cor:wt-controls-colour-type}
For $\s,\t \in \SSTab(\kl; m)$, we have $\wt(\s) = \wt(\t)$ if and only if $\s$ and $\t$ have the same colour type.
\end{cor}

We next deal with the case $s=1$.

\begin{thm} \label{thm:klm}
We have
$$
f_{\upa{(k,\ell)}{m}} = \sum_{\s \in \SSTab((k,\ell);m)} \frac{1}{\binom{k-\ell+1+\wt(\s)}{\wt(\s)}} \, e_{\s}.
$$
Here we identify $\BW^{\min(\ell,m)}_1$ with $[0,\min(\ell,m)]$, so that $\wt(\s) \in \mathbb{Z}_{\geq 0}$ for all $\s \in \SSTab((k,\ell); m)$.
\end{thm}

\begin{proof}
First, for $x,y \in \ZZ^+$ with $x \geq y$, and $i \in [0,y]$, let $\uu^{x,y}_i$ be the standard $(x+1,y)$-tableau whose first row contains $[1,x]$ and $x+1+i$.  Then $\SSTab((x,y);1) = \{ \uu^{x,y}_i \mid i \in [0,y] \}$ and
$$
\wt(\uu^{x,y}_i) =
\begin{cases}
0, &\text{if } i = y; \\
1, &\text{if } i \in [0,y-1].
\end{cases}
$$

We prove by induction on $m$.  For $m=1$, we have
\begin{align*}
f_{\upa{(k,\ell)}{1}} =  e_{\uu^{k,\ell}_{\ell}} + \tfrac{1}{k-\ell+2} \sum_{i=0}^{\ell-1} e_{\uu^{k,\ell}_i},
\end{align*}
by Theorem \ref{thm:mu=(1)}, agreeing with the theorem here.

For $m>1$, we have
$$f_{\upa{(k,\ell)}{m}}
= \DD{\upa{(k,\ell)}{m-1}} \DD{\upa{(k+m-1,\ell)}{1}} e_{\IT^{(k+m,\ell)}}$$
 by Corollary \ref{cor:same-relative-positions}(2).  By induction,
\begin{align*}
\DD{\upa{(k,\ell)}{m-1}} &= \sum_{\tilde{\s} \in \SSTab((k,\ell);m-1)} \tfrac{1}{\binom{k-\ell+1+\wt(\tilde{\s})}{\wt(\tilde{\s})}} \,d(\tilde{\s}), \\
\DD{\upa{(k+m-1,\ell)}{1}} &= \sum_{\uu \in \SSTab((k+m-1,\ell);1)} \tfrac{1}{\binom{k+m-\ell+\wt(\uu)}{\wt(\uu)}} d(\uu) \\
&= \sum_{i = 0}^{\ell-1} \tfrac{1}{k+m-\ell+1} d(\uu^{k+m-1,\ell}_i) + d(\uu^{k+m-1,\ell}_{\ell}).
\end{align*}
Thus,
\begin{align*}
f_{\upa{(k,\ell)}{m}}
&= \sum_{\tilde{\s} \in \SSTab((k,\ell);m-1)} \tfrac{1}{\binom{k-\ell+1+\wt(\tilde{\s})}{\wt(\tilde{\s}) }} \left(\sum_{i=0}^{\ell-1} \tfrac{1}{k+m-\ell+1} e_{d(\tilde{\s}) \cdot \uu^{k+m-1,\ell}_i} + e_{d(\tilde{\s}) \cdot \uu^{k+m-1,\ell}_{\ell}} \right).
\end{align*}
Note that for each $\tilde{\s} \in \SSTab((k,\ell);m-1)$ and $i \in [0,\ell]$, $\overline{d(\tilde{\s}) \cdot \uu^{k+m-1,\ell}_i} \in \STab((k+m,\ell))$ and its first row contains the first row of $\tilde{\s}$, and that $k+\ell+m$ appears in the same node as in both $\uu^{k+m-1,\ell}_i$ and $\overline{d(\tilde{\s}) \cdot \uu^{k+m-1,\ell}_i}$.
Thus, for $w \in [0,\min(\ell,m)]$ with $w < m$, we have $e_{d(\tilde{\s}) \cdot \uu^{k+m-1,\ell}_i} = e_{\s^{k,\ell,m}_w}$ if and only if $\tilde{\s} = \s^{k,\ell,m-1}_w$ and $i = \ell$, so that the coefficient of $e_{\s^{k,\ell,m}_w}$ appearing in $f_{\upa{(k,\ell)}{m}}$ is $\tfrac{1}{\binom{k-\ell+1+\wt(\tilde{\s})}{\wt(\tilde{\s}) }} = \tfrac{1}{\binom{k-\ell+1+w}{w}}$.
On the other hand, if $m \in [0,\min(\ell,m)]$ (so $\ell \geq m$), then
$\IT^{(k+m,\ell)} \in \SSTab((k,\ell);m)$ with $\wt(\IT^{(k+m,\ell)}) = m$, and
$e_{d(\tilde{\s}) \cdot \uu^{k+m-1,\ell}_i} = e_{\IT^{(k+m,\ell)}}$ if and only if
$$[1,k+m] = d(\tilde{\s}) d(\uu^{k+m-1,\ell}_i) ([1,k+m]) = d(\tilde{\s}) ([1,k+m-1] \cup \{k+m+i\}).$$
Since $\IT^{(k+m-1,\ell)}(1,j) = j$ for all $j \in [1,k+m-1]$ and $\IT^{(k+m-1,\ell)}(2,i+1) = k+m+i$, we have
\begin{align*}
&d(\tilde{\s}) ([1,k+m-1] \cup \{k+m+i\}) \\
&= d(\tilde{\s}) (\{ \IT^{(k+m-1,\ell)}(1,j) \mid j \in [1,k+m-1] \} \cup \{ \IT^{(k+m-1,\ell)}(2,i+1) \}) \\
&= \{ \tilde{\s}(1,j) \mid j \in [1,k+m-1] \} \cup \{ \tilde{\s}(2,i+1) \}.
\end{align*}
Thus, $[1,k+m] = d(\tilde{\s}) ([1,k+m-1] \cup \{k+m+i\})$ only if $i =0$ (otherwise $\tilde{\s}(2,1) \notin [1,k+m]$ while $\tilde{\s}(2,i+1) \in [1,k+m]$, contradicting the (row) standardness of $\tilde{\s}$), in which case, we can choose any $\tilde{\s} \in \SSTab((k,\ell);m-1)$ whose first row contains both $[1,k]$ and an $(m-1)$-element subset of $[k+1,k+m]$; all of such tableaux have weight $m-1$.
Hence, the coefficient of $e_{\IT^{(k+m,\ell)}}$ appearing in $f_{\upa{(k,\ell)}{m}}$ is $\tfrac{m}{\binom{k-\ell+m}{m-1 }(k+m-\ell+1)} = \tfrac{1}{\binom{k-\ell+1+m}{m}}$.  By Theorem \ref{lem:semistandard}(1) and Corollary \ref{cor:wt-controls-colour-type}, we conclude that
$$
f_{\upa{(k,\ell)}{m}} = \sum_{\s \in \SSTab((k,\ell);m)} \tfrac{1}{\binom{k-\ell+1+\wt(\s)}{\wt({\s})}} e_{\s}
$$
as desired.
\end{proof}

To obtain a closed formula for $\dd_{\upa{(k,\ell)}{m}}$, we need the following:

\begin{lem} \label{lem:binomial}
Let $a \in \ZZ^+$ and $b \in \ZZ_{\geq 0}$.  Then \[\lcm\left\{ \binom{a+r}{a} \mid r \in [0, b] \right\} = \frac{\lcm[a,\, a+b]}{a}.\]
\end{lem}

\begin{proof}
Multiplying the required equality by $a$ throughout, we get the equivalent statement $\lcm\{ a\binom{a+r}{a} \mid r \in [0,b] \} = \lcm[a,a+b]$, which we shall prove by induction on $b$, with $b=0$ being trivial.
Assume thus $b>0$, and that $\lcm\{ a\binom{a+r}{a} \mid r \in [0, b-1] \} = \lcm[a,a+b-1]$.
Then $\lcm\{ a\binom{a+r}{a} \mid r \in [0, b] \} = \lcm \{ \lcm[a,a+b-1], a\binom{a+b}{a} \}$.
By \cite[Lemma 4.12(1)]{FLT}, we have $a\binom{a+b}{a} \mid \lcm[a,a+b]$, so that
$$\lcm \{ \lcm[a,a+b-1], a\tbinom{a+b}{a} \} \mid \lcm[a,a+b].$$
On the other hand, $(a+b) \mid (a+b) \binom{a+b-1}{b} = a\binom{a+b}{a}$, so that
$$
\lcm[a,a+b] = \lcm \{ \lcm[a,a+b-1], a+b\} \mid \lcm \{ \lcm[a,a+b-1], a\tbinom{a+b}{a} \}.$$  Thus
$$
\lcm\{ a\tbinom{a+r}{a} \mid r \in [0, b] \} = \lcm \{ \lcm[a,a+b-1], a\tbinom{a+b}{a} \} = \lcm [a,a+b],$$
as desired.
\end{proof}

\begin{cor} \label{cor:denominator-klm}
We have
$$
\dd_{\upa{(k,\ell)}{m}} = \frac{\lcm[k -\ell+1, k-\ell +1+ \min(\ell,m)]}{k-\ell+1}.
$$
\end{cor}

\begin{proof}
This follows from Theorem \ref{thm:klm} and Lemma \ref{lem:binomial} immediately.
\end{proof}

\begin{eg}
Let $k= 3$, $\ell = m = 2$.  The following is a complete list of $\s_w \in \SSTab((3,2);2)$ for each $w \in \BW_1^{\min\{ \ell, m \}} = [0,2]$:
\[
\begin{array}{c|ccc}
w & 0 & 1 & 2 \\ \\
\s_w &
\ytableausetup{notabloids}
\raisebox{1mm}{\begin{ytableau}
\scriptstyle 1 & \scriptstyle 2 & \scriptstyle 3 & \scriptstyle \textcolor{blue}{6} & \scriptstyle \textcolor{blue}{7} \\
\scriptstyle \textcolor{red}{4} & \scriptstyle \textcolor{red}{5}
\end{ytableau}}
&
\raisebox{1mm}{\begin{ytableau}
\scriptstyle 1 & \scriptstyle 2 & \scriptstyle 3 & \scriptstyle \textcolor{red}{5} & \scriptstyle \textcolor{blue}{7}  \\
\scriptstyle \textcolor{red}{4} & \scriptstyle \textcolor{blue}{6}
\end{ytableau}}
&
\raisebox{1mm}{\begin{ytableau}
\scriptstyle 1 & \scriptstyle 2 & \scriptstyle 3 & \scriptstyle \textcolor{red}{4} & \scriptstyle \textcolor{red}{5} \\
\scriptstyle \textcolor{blue}{6} & \scriptstyle \textcolor{blue}{7}
\end{ytableau}}
\end{array}
\]
By Theorem \ref{thm:klm}, $f_{\upa{(3,2)}{2}} = \sum_{\s \in \SSTab((5,2))} \frac{2}{(\wt(\s) + 2)(\wt(\s) + 1)} e_{\s}$.
Hence $$\dd_{\upa{(3,3)}{3}} = \lcm \{ \tfrac{(w+2)(w+1)}{2} \mid w \in [0,2] \} = \lcm \{1,3,6\} = 6 = \tfrac{\lcm [2,4]}{2}, $$ agreeing with Corollary \ref{cor:denominator-klm}.
\end{eg}

\begin{rem}
By \cite[Theorem 3.13]{FLT}, Corollary \ref{cor:denominator-klm} can be used to provide a simpler proof to
  \cite[Theorem 4.13]{FLT}.
\end{rem}

For the remainder of this section, write $\ba^{k,\ell}_{m,\bw}$ for $q_{\upa{(k,\ell^{s})}{m},\s_{\bw}}$ for each $\bw \in \BW^{\min(\ell,m)}_s$, so that
$$
f_{\upa{(k,\ell^s)}{m}} = \sum_{\bw \in \BW^{\min(\ell,m)}_s} \ba^{k,\ell}_{m,\bw} \Ew{\bw},
$$
where $\Ew{\bw} = \sum_{\substack{\s \in \SSTab(\kl;m) \\ \wt(\s) = \bw }} e_{\s}$,
by Theorem \ref{lem:semistandard} and Corollary \ref{cor:wt-controls-colour-type}.  Equivalently,
$$\DD{\upa{(k,\ell^{s})}{m}} = \sum_{\s \in \SSTab((k,\ell^{s});m)} \ba^{k,\ell}_{m,\wt(\s)} d(\s).$$
We shall provide four reduction results, relating $\ba^{k,\ell}_{m,\bw}$ to others with different parameters.

To state the first result, note that we have a natural left action of $\sym{s}$ on $\BW^{\min(\ell,m)}_s$ via place permutations: $\sigma \cdot (w_1,w_2,\dotsc, w_s) = (w_{\sigma^{-1}(1)}, w_{\sigma^{-1}(2)}, \dotsc, w_{\sigma^{-1}(s)})$ for all $\sigma \in \sym{s}$ and $(w_1,w_2,\dotsc, w_s) \in \BW^{\min(\ell,m)}_s$.

\begin{thm} \label{thm:4reductions}
Let $\bw\in \BW^{\min(\ell,m)}_s$.
We have:
\begin{enumerate}[itemsep=5pt]
\item $|\ba^{k,\ell}_{m,\bw}| = |\ba^{k,\ell}_{m,\sigma \cdot \bw}|$ for all $\sigma \in \sym{s}$;

\item $\ba^{k,\ell}_{m,(0,\bw)} = \ba^{k+1,\ell}_{m,\bw}$ if $s \geq \min(\ell,m)$;

\item $\ba^{k,\ell}_{m,\bw} = \ba^{k,\ell}_{m-1,\bw}$ if $m> \ell$;

\item $\ba^{k,\ell}_{m,\bw} = \ba^{k-1,\ell-1}_{m,\bw}$ if $\ell > m$.
\end{enumerate}
\end{thm}

\begin{proof} For each part we fix some of the parameters $k,\ell,m,s$, and will omit them in the notations to make the latter less cumbersome.
\begin{enumerate}
\item We fix $k,\ell,m,s$ and prove $|\ba_{\bw}| = |\ba_{\sigma\cdot \bw}|$ for all $\sigma \in \sym{s}$ and $\bw \in \BW$.  In fact, it suffices to prove this for $\sigma = \bs_i \in \sym{s}$.  Let $$\theta_i = \prod_{j=1}^\ell (k+(i-1)\ell+j, k+i\ell+j) \in \sym{[k+(i-1)\ell+1,\, k+(i+1)\ell]}.$$
Observe that the effect of $\theta_i$ on $\IT^{(k,\ell^s)}$ is to swap its $(i+1)$-th and $(i+2)$-th rows.
Thus,
$$
\theta_i f_{\IT^{(k,\ell^s)}} = \theta_i e_{\IT^{(k,\ell^s)}} = e_{\theta_i \cdot \IT^{(k,\ell^s)}} = (-1)^{\ell} e_{\IT^{(k,\ell^s)}} = (-1)^{\ell} f_{\IT^{(k,\ell^s)}},$$
where the third equality follows from \eqref{E:Garnir-1}.
This yields $$\theta_i f_{\upa{(k,\ell^s)}{m}} = (-1)^{\ell} f_{\upa{(k,\ell^s)}{m}}$$ by Corollary \ref{cor:same-relative-positions}(1).

On the other hand, $\theta_i f_{\upa{(k,\ell^s)}{m}} = \sum_{\bw \in \BW} \ba_{\bw} \theta_i \Ew{\bw}$.
For each $\bw =(w_1,\dotsc,w_s) \in \BW$, let $S^{\QQ}_{(k+m,\ell^s)}(\bw)$ be the $\QQ$-span of the polytabloids $e_{\s}$ which are labelled by $\s \in \STab((k+m,\ell^s))$ whose first row contains all the $k$ integers with colour $c_1$, and exactly $w_i$ integers with colour $c_{i+1}$ for all $i \in [1,s]$.
Then since $e_{\theta_i \cdot \s} \in S^{\QQ}_{(k+m,\ell^s)}(\bs_i \cdot \bw)$ for any $\s \in \SSTab((k,\ell^s);m)$ with $\wt(\s) = \bw$ by Proposition \ref{prop:Garnir}(3) (with $l=1$),
we have
$\theta_i \Ew{\bw} = \theta_i \sum e_{\s} = \sum e_{\theta_i \cdot \s} \in S^{\QQ}_{(k+m,\ell^s)}(\bs_i \cdot \bw)$.
Thus, comparing the $S^{\QQ}_{(k+m,\ell^s)}(\bs_i \cdot \bw)$-components on both sides of
$$
(-1)^\ell \sum_{\bw \in \BW} \ba_{\bw} \Ew{\bw} = (-1)^{\ell} f_{\upa{(k,\ell^s)}{m}} = \theta_i f_{\upa{(k,\ell^s)}{m}} = \sum_{\bw \in \BW} \ba_{\bw} \theta_i \Ew{\bw},
$$
we have $(-1)^{\ell} \ba_{\bs_i \cdot \bw} \Ew{\bs_i \cdot \bw} = \ba_{\bw} \theta_i \Ew{\bw}$.
In particular, $\ba_{\bs_i \cdot \bw} \ne 0$ if and only if $\ba_{\bw} \ne 0$.
Suppose that $\ba_{\bw} \ne 0$, and let $\ba_{i,\bw} = (-1)^\ell \frac{\ba_{\bs_i \cdot \bw}}{\ba_{\bw}}$.
Then $\ba_{i,\bw} \in \QQ \setminus \{0\}$, and $\theta_i \Ew{\bw} = \ba_{i,\bw} \Ew{\bs_i \cdot \bw}$.  Since $\theta_i^2 = 1$, we also have $ \theta_i \Ew{\bs_i \cdot \bw} = \frac{1}{\ba_{i,\bw}} \Ew{\bw}$.
But $\theta_i \Ew{\bw}, \theta_i \Ew{\bs_i\cdot \bw} \in S^{\ZZ}_{(k+m,\ell^s)}$, so this forces $\ba_{i,\bw}, \frac{1}{\ba_{i,\bw}} \in \ZZ$, i.e.\ $\ba_{i,\bw} = \pm 1$.  Hence, $|\ba_{\bw}| = |\ba_{\bs_i \cdot \bw}|$ as required.

\item
We fix $\ell, m$ and prove $\ba^{k}_{(0,\bw)} = \ba^{k+1}_{\bw}$ for all $\bw \in \BW_s$ when $s \geq \min(\ell,m)$; here $(0,\bw) = (0,w_1,\dotsc, w_s) \in \BW_{s+1}$ when $\bw = (w_1,\dotsc, w_s)$. Let
\begin{alignat*}{2}
\t &= \upI{(k,\ell)}{(k+m,\ell^{s+1})}, &\qquad \t' &= \IT^{(k+1+m,\ell^{s})},\\
\s &= \upa{(k,\ell^{s+1})}{m},& \qquad \s' &= \upa{(k+1,\ell^{s})}{m}.
\end{alignat*}
By Corollary \ref{cor:same-relative-positions}(3) (with $i=2$),
we have
$
f_{\s} = \DD{\s'}^{+(\ell-1)} f_{\t}.
$
By Lemma \ref{lem:f_in_terms_of_e}(1), we have $f_{\t} = e_{\t} + \sum_{\rr \rhd \t} q_{\t,\rr} e_{\rr}$.
Thus,
\begin{equation} \label{E:2nd-reduction}
f_{\s} = \DD{\s'}^{+(\ell-1)} f_{\t} = \DD{\s'}^{+(\ell-1)} e_{\t} + \sum_{\rr \rhd \t} q_{\t,\rr} \DD{\s'}^{+(\ell-1)} e_{\rr}.
\end{equation}
By Lemma \ref{lem:t^lambda-uparrow^nu}, for each $\rr \rhd \t$, we have $\Shape(\rr {\downarrow_{k+\ell}}) \rhd (k,\ell)$, so that $\rr$ contains $[1,k]$ and some integer $i_{\rr} \in [k+1,k+\ell]$ in its first row. 
Consequently, since for all $\tau \in \sym{[k+\ell+1,k+m+ (s+1)\ell]}$, $e_{\tau \cdot \rr}$ is spanned by polytabloids indexed by standard tableaux having the same first row as $\tau \cdot \rr$, which contains the integer $i_{\rr}$ with colour $c_2$, by Proposition \ref{prop:Garnir}(3) (with $l=1$), the same holds for $\DD{\s'}^{+(\ell-1)} e_{\rr}$ as $\DD{\s'}^{+(\ell-1)} \in \QQ \sym{[k+\ell+1,k+m+ (s+1)\ell]}$ by Lemma \ref{lem:f_in_terms_of_e}(2).
The upshot is that, in \eqref{E:2nd-reduction}, only $\DD{\s'}^{+(\ell-1)} e_{\t}$ contributes to the coefficient $\ba^k_{(0,\bw)}$ of $e_{\s^k_{(0,\bw)}}$ in $f_\s$.
Now,
$$
\DD{\s'}^{+(\ell-1)}e_{\t}
= \sum_{\uu \in \SSTab((k+1,\ell^s);m)} \ba^{k+1}_{\wt(\uu)} e_{d(\uu)^{+(\ell-1)} \cdot \t}.
$$
As $d(\uu)^{+(\ell-1)} \cdot \t$ is standard for all $\uu \in \SSTab((k+1,\ell^s);m)$ and the map $\uu \mapsto d(\uu)^{+(\ell-1)} \cdot \t$ is injective, the coefficient of $e_{\s^k_{(0,\bw)}}$ in $\DD{\s'}^{+(\ell-1)} e_{\t}$ is $\ba^{k+1}_{\wt(\uu)}$ where $d(\uu)^{+(\ell-1)} \cdot \t = \s^{k}_{(0,\bw)}$, which is precisely where $\uu = \s^{k+1}_{\bw}$. Thus $\ba^k_{(0,\bw)} = \ba^{k+1}_{\bw}$ as desired.

\item
We fix $k,\ell, s$ and prove $\ba_{m,\bw} = \ba_{m-1,\bw}$ for all $\bw \in \BW^{\ell}_s$ when $m >\ell$.
By Lemma \ref{lem:f_in_terms_of_e}(1,2) and Corollary \ref{cor:same-relative-positions}(2), $\DD{\upa{(k,\ell^s)}{m-1}} \in \QQ\sym{[k+1,k+m-1+s\ell]}$ and
\begin{align*}
f_{\upa{(k,\ell^s)}{m}}
&= \DD{\upa{(k,\ell^s)}{m-1}} f_{\upa{(k+m-1,\ell^s)}{1}} \\
&= \DD{\upa{(k,\ell^s)}{m-1}} (e_{\upa{(k+m-1,\ell^s)}{1}} + \sum_{\rr \rhd \upa{(k+m-1,\ell^s)}{1}} q_{\upa{(k+m-1,\ell^s)}{1},\rr} e_{\rr}).
\end{align*}
If $\rr \rhd \upa{(k+m-1,\ell^s)}{1}$, then $\Shape(\rr {\downarrow_{k+m-1+s\ell}}) \rhd (k+m-1,\ell^s)$
by Lemma \ref{lem:t^lambda-uparrow^nu}, so that $\Shape(\rr {\downarrow_{k+m-1+s\ell}}) = (k+m,\ell^{s-1},\ell-1)$. 
Thus, for any $\tau \in \sym{k+m-1+s\ell}$, we have $\Shape((\overline{\tau \cdot \rr}) {\downarrow_{k+m-1+s\ell}}) = (k+m,\ell^{s-1},\ell-1)$ so that $e_{\tau \cdot \rr}$ is spanned by standard $(k+m, \ell^s)$-polytabloids $e_{\uu}$ such that $\Shape(\uu {\downarrow_{k+m-1+s\ell}}) = (k+m,\ell^{s-1},\ell-1)$ by Proposition \ref{prop:Garnir}(2).
Consequently, since $\DD{\upa{(k,\ell^s)}{m-1}} \in \QQ\sym{k+m-1+s\ell}$, $\DD{\upa{(k,\ell^s)}{m-1}} e_{\rr}$ lies in the span of polytabloids indexed by standard tableaux in which $k+m+s\ell$ lie in their respective last row.
Since for all $\bw \in \BW^{\ell}_s$, $\s^{m}_{\bw}$ contains $k+m+s\ell$ in its first row for all $\bw \in \BW^\ell_s$, we see that $\DD{\upa{(k,\ell^s)}{m-1}} e_{\rr}$ does not contribute to the coefficient $\ba_{m,\bw}$ of $e_{\s^{m}_{\bw}}$ in $f_{\upa{(k,\ell^s)}{m}}$.
On the other hand,
\begin{align*}
\DD{\upa{(k,\ell^s)}{m-1}}\, e_{\upa{(k+m-1,\ell^s)}{1}}
&= \sum_{\tilde{\s} \in \SSTab((k,\ell^s),m-1)} \ba_{m-1,\wt(\tilde{\s})} e_{d(\tilde{\s}) \cdot \IT^{(k+m,\ell^s)}} \\
&= \sum_{\tilde{\s} \in \SSTab((k,\ell^s),m-1)} \ba_{m-1,\wt(\tilde{\s})} e_{\tilde{\s}\uparrow^{(k+m,\ell^s)}}.
\end{align*}
Clearly, $\tilde{\s} {\uparrow^{(k+m,\ell^s)}}$ is standard for all $\tilde{\s} \in \SSTab((k,\ell^s);m-1)$, and the map $\tilde{\s} \mapsto \tilde{\s} {\uparrow^{(k+m,\ell^s)}}$ is injective.
Furthermore, $\s^{m-1}_{\bw} {\uparrow^{(k+m,\ell^s)}} = \s^{m}_{\bw}$.
Thus $\ba_{m-1,\bw} = \ba_{m,\bw}$ just as in part (2).

\item
We fix $m$ and prove that $\ba^{k,\ell}_{\bw_s} = \ba^{k-1,\ell-1}_{\bw_s}$ for all $\bw_s \in \BW^{m}_s$ when $\ell > m$.

We prove by induction on $s$.  For $s=1$  we have $\ba^{k,\ell}_{(w)} = \frac{1}{\binom{k-\ell+w+1}{w}} =  \ba^{k-1,\ell-1}_{(w)}$ by Theorem \ref{thm:klm}, as desired.
Assume thus $s>1$.

By Corollary \ref{cor:same-relative-positions}(3) (with $i=2$),
\begin{alignat}{2}
f_{\upa{(k,\ell^s)}{m}} &
= \DD{\upa{(k+1,\ell^{s-1})}{m}}^{+(\ell-1)} \DD{\upa{(k,\ell)}{m}}\, e_{\IT^{(k+m,\ell^s)}} \notag \\
&= \sum_{\substack{\uu \in \SSTab((k+1,\ell^{s-1});m) \\ \s \in \SSTab((k,\ell);m)}} \ba^{k+1,\ell}_{\wt(\uu)} \, \ba^{k,\ell}_{\wt(\s)} \, d(\uu)^{+(\ell-1)} d(\s)\,  e_{\IT^{(k+m,\ell^s)}} \notag \\
&= \sum_{\substack{\uu \in \SSTab((k+1,\ell^{s-1});m) \\ \s \in \SSTab((k,\ell);m)}} \ba^{k+1,\ell}_{\wt(\uu)} \, \ba^{k,\ell}_{\wt(\s)} \, e_{\s \uparrow^\uu} \label{E:4th-reduction}
\end{alignat}
where $\s {\uparrow^{\uu}} = (d(\uu)^{+(\ell-1)} d(\s)) \cdot \IT^{(k+m,\ell^s)}$. (Note that \eqref{E:4th-reduction} holds for all $k,\ell,m,s \in \ZZ^+$ with $k \geq \ell$.)  Since $d(\s) \cdot \IT^{(k+m,\ell^s)} = \s {\uparrow^{(k+m,\ell^s)}}$, $\s {\uparrow^{\uu}}$ has
the following properties:

\begin{itemize}[topsep=3pt]
\item its first row contains $[1,k]$ and exactly $\wt(\s)$ integers in $[k+1,k+\ell]$;
\item its third and subsequent rows are exactly those of the second and subsequent rows of $\uu$ translated by $(\ell-1)$.
\end{itemize}

Let $\mathfrak{T}_{k,\ell,s} = \{ \t \in \RSTab((k+m,\ell^s)) \mid \t(i+1,1) = k+(i-1)\ell +1\ \forall i \in [1,s] \}$, and let $\mathfrak{T}^{\STab}_{k,\ell,s} = \mathfrak{T}_{k,\ell,s} \cap \STab((k+m,\ell^s))$ and $\mathfrak{T}^{\SSTab}_{k,\ell,s}  = \mathfrak{T}_{k,\ell,s}\cap \SSTab((k,\ell^s);m)$.
Since $\ell > m$, we have $\s^{k,\ell}_{\bw_s} \in \mathfrak{T}^{\SSTab}_{k,\ell,s}$ for all $\bw_s \in \BW^m_s$.

Let
$$
\xi_{k,\ell,s} : [2,k+s\ell+m]\setminus \{ k+(i-1)\ell +1 \mid i \in [1,s] \}  \to [1,k-1+s(\ell-1) + m]
$$
denote the order-preserving bijection between these two subsets of $\ZZ^+$ with the same cardinality.
We have an injection
$\Xi_{k,\ell,s} : \mathfrak{T}_{k,\ell,s} \to \RSTab((k-1+m,(\ell-1)^s))$ defined by
$\Xi_{k,\ell,s}(\s) (i,j) = \xi_{k,\ell,s}(\s(i,j+1))$ for all $(i,j) \in [(k-1+m,(\ell-1)^s)]$,
that preserves standardness, colour-semistandardness and weight (where $[k+(i-1)\ell+1,k+i\ell]$ for the tableaux in $\mathfrak{T}_{k,\ell,s}$  and $[k+(i-1)(\ell-1), k-1+ i(\ell -1)]$ for the tableaux in $\RSTab((k-1+m,(\ell-1)^s))$ are coloured $c_{i+1}$).
Note that for $\bw_s \in \BW^m_s$, we have $\Xi_{k,\ell,s} ( \s^{k,\ell}_{\bw_s}) = \s^{k-1,\ell-1}_{\bw_s}$.

Let $\pi_{\mathfrak{T}^{\STab}_{k,\ell,s}} : S_{(k+m,\ell^s)}^\QQ \to S_{(k-1+m,(\ell-1)^s)}^\QQ$ be the $\QQ$-linear map defined by
$$
e_{\s} \mapsto
\begin{cases}
e_{\Xi_{k,\ell,s}(\s)}, &\text{if } \s \in \mathfrak{T}^{\STab}_{k,\ell,s}; \\
0, &\text{otherwise}
\end{cases}
$$
for all $\t \in \STab((k+m,\ell^s))$.
We claim that $\pi_{\mathfrak{T}^{\STab}_{k,\ell,s}}(e_{\t}) = e_{\Xi_{k,\ell,s}(\t)}$ for all $\t \in \mathfrak{T}_{k,\ell,s}$, which is of course the definition of $\pi_{\mathfrak{T}^{\STab}_{k,\ell,s}}$ if $\t$ is standard.
When $\t \in \mathfrak{T}_{k,\ell,s}$ is not standard, then there exist $(i,j), (i+1,j) \in [(k+m,\ell^s)]$ such that $\t(i,j) > \t(i+1,j)$.
Let $X = \{ \t(i, a) \mid a \in [j,\ell] \}$ and $Y = \{ \t(i+1,b) \mid b \in [1,j] \}$.
Then by Proposition \ref{prop:Garnir}(1), we have
$$
e_{\t} = - \sum_{\sigma \in G_{X,Y} \setminus \{1\}} e_{\sigma \cdot \t}.
$$
Let $y = \t(i+1,1)$, and let $Y' = Y \setminus \{ y \}$.  We may choose $G_{X,Y}$ so that it contains $G_{X,Y'}$.
If $\sigma \in G_{X,Y}$ and $\sigma^{-1}(y) \notin X$, then $y' := \sigma^{-1}(y) \in Y$, and so $\sigma (y,y') \in \sym{X \cup Y'}$.
Thus $\sigma (y,y') = \sigma'\tau$ for some $\sigma' \in G_{X,Y'}$ and $\tau \in \sym{X}\sym{Y'}$; in particular $\sigma\sym{X}\sym{Y} = \sigma' \sym{X}\sym{Y}$ and hence $\sigma = \sigma' \in G_{X,Y'}$.
Therefore, if $\sigma \in G_{X,Y} \setminus G_{X,Y'}$, we have $\sigma^{-1}(y) \in X$ and so $\sigma \cdot \t$ contains $y = \t(i+1,1) =  k+(i-1)\ell+1$ in its $i$-th row.  Consequently, $\pi_{\mathfrak{T}^{\STab}_{k,\ell,s}} (e_{\sigma \cdot \t}) = 0$ by Proposition \ref{prop:Garnir}(2), since $\s \ntrianglerighteq \overline{\sigma \cdot \t}$ for all $\s \in \mathfrak{T}_{k,\ell,s}^{\STab}$.
Thus,
$$
\pi_{\mathfrak{T}^{\STab}_{k,\ell,s}}(e_{\t}) = - \sum_{\sigma \in G_{X,Y'} \setminus \{1\}} \pi_{\mathfrak{T}^{\STab}_{k,\ell,s}}(e_{\sigma \cdot \t}),
$$
and since $\t \lhd \overline{\sigma \cdot \t} \in \mathfrak{T}_{k,\ell,s}$ for all $\sigma \in G_{X,Y'} \setminus \{1\}$ by Proposition \ref{prop:Garnir}(1), we conclude by induction that
$$
\pi_{\mathfrak{T}^{\STab}_{k,\ell,s}}(e_{\t}) = - \sum_{\sigma \in G_{X',Y} \setminus \{1\}} e_{\Xi_{k,\ell,s}(\overline{\sigma \cdot \t})}.
$$
Now note that:
\begin{itemize}[topsep=3pt, itemsep=3pt]
\item $\Xi_{k,\ell,s}(\overline{\sigma \cdot \t}) = \overline{\xi\sigma\xi^{-1} \cdot \Xi_{k,\ell,s}(\t)}$;
\item $\sigma \in G_{X,Y'}$ if and only if $\xi\sigma\xi^{-1} \in \xi G_{X,Y'} \xi^{-1}$ and we may choose $G_{\xi(X),\xi(Y')}$ to be $\xi G_{X,Y'}\xi^{-1}$;
\item $\xi(Y') = \{ \xi(\t(i+1,b)) \mid b \in [2,j] \} = \{ \Xi_{k,\ell, s}(\t)(i+1,b) \mid b \in [1,j-1] \}$, and similarly, $\xi(X) = \{ \Xi_{k,\ell,s}(\t)(i, a) \mid a \in [j-1,\ell-1] \}$.
\end{itemize}
Thus,
$$
\pi_{\mathfrak{T}^{\STab}_{k,\ell,s}}(e_{\t}) = -\sum_{\tau \in G_{\xi(X'), \xi(Y) \setminus \{1\}}} e_{\tau \cdot\, \Xi_{k,\ell,s}(\t)}
= e_{\Xi_{k,\ell,s}(\t)}
$$
by Proposition \ref{prop:Garnir}(1), establishing our claim.

We now investigate $\pi_{\mathfrak{T}_{\STab}}(e_{\s \uparrow^{\uu}})$ where $\s \in \SSTab((k,\ell);m)$ and $\uu \in \SSTab((k+1, \ell^{s-1});m)$.
From the properties of $\s {\uparrow^{\uu}}$ listed above, we see that if $\overline{\s {\uparrow^{\uu}}} \notin \mathfrak{T}_{k,\ell,s}$, then $k+(i-1)\ell+1$ lies above its $(i+1)$-th row for some $i \in [1,s]$ (by Lemma \ref{lem:easy-lem-for-semistd} applied to $\uu$), and hence $\t \ntrianglerighteq \overline{\s {\uparrow^{\uu}}} $ for all $\t \in \mathfrak{T}^{\STab}_{k,\ell,s}$, so that $\pi_{\mathfrak{T}^{\STab}_{k,\ell,s}}(e_{\s {\uparrow^{\uu}}}) = 0$ by Proposition \ref{prop:Garnir}(2).
On the other hand, if $\overline{\s {\uparrow^{\uu}}} \in \mathfrak{T}_{k,\ell,s}$, we have, by the claim in the last paragraph, that $\pi_{\mathfrak{T}^{\STab}_{k,\ell,s}}(e_{\s \uparrow^{\uu}}) = e_{\Xi_{k,\ell,s}(\overline{\s \uparrow^{\uu}})}$.
Now note that:
\begin{itemize}[topsep=3pt, itemsep=3pt]
\item for $\s \in \SSTab((k,\ell);m)$ and $\uu \in \SSTab((k+1,\ell^{s-1});m)$, we have $\overline{\s {\uparrow^{\uu}}} \in \mathfrak{T}_{k,\ell,s}$ if and only if $\s \in \mathfrak{T}_{k,\ell,1}$ and $\uu \in \mathfrak{T}_{k+1,\ell,s-1}$, in which case $\Xi_{k,\ell,s}(\overline{\s {\uparrow^{\uu}}}) = \overline{\Xi_{k,\ell,1}(\s) {\uparrow^{\Xi_{k+1,\ell,s-1}(\uu)}}}$;
\item $\Xi_{k,\ell,1}$ maps $\mathfrak{T}^{\SSTab}_{k,\ell,1}$ bijectively onto $\SSTab((k-1,\ell-1);m)$, and is weight-preserving;
\item $\Xi_{k+1,\ell, s-1}$ maps $\mathfrak{T}^{\SSTab}_{k+1,\ell,s-1}$ bijectively onto $\SSTab((k,(\ell-1)^{s-1});m)$, and is weight-preserving.
\end{itemize}

Thus,
\begin{align*}
\pi_{\mathfrak{T}^{\STab}_{k,\ell,s}}(f_{\upa{(k,\ell^s)}{m}})
&= \sum_{\substack{\uu \in \SSTab((k+1,\ell^{s-1});m) \\ \s \in \SSTab((k,\ell);m)}} \ba^{k+1, \ell}_{\wt(\uu)} \ba^{k,\ell}_{\wt(\s)} \pi_{\mathfrak{T}^{\STab}_{k,\ell,s}} (e_{\s \uparrow^{\uu}}) \\
&= \sum_{\substack{\uu \in \mathfrak{T}^{\SSTab}_{k+1,\ell, s-1} \\ \s \in \mathfrak{T}^{\SSTab}_{k,\ell, 1}}} \ba^{k+1, \ell}_{\wt(\uu)} \ba^{k,\ell}_{\wt(\s)} \, e_{\Xi_{k,\ell,1}(\s) \uparrow^{\Xi_{k+1,\ell,s-1}(\uu)}} \\
&= \sum_{\substack{\uu' \in \SSTab((k,(\ell-1)^{s-1});m) \\ \s' \in \SSTab((k-1,\ell-1);m)}} \ba^{k+1, \ell}_{\wt(\uu')} \ba^{k,\ell}_{\wt(\s')} \, e_{\s' \uparrow^{\uu'}} \\
&= \sum_{\substack{\uu' \in \SSTab((k,(\ell-1)^{s-1});m) \\ \s' \in \SSTab((k-1,\ell-1);m)}} \ba^{k, \ell-1}_{\wt(\uu')} \ba^{k-1, \ell-1}_{\wt(\s')}\, e_{\s' \uparrow^{\uu'}} \\
&= f_{\upa{(k-1,(\ell-1)^s)}{m}} = \sum_{\vv \in \SSTab((k-1,(\ell-1)^s);m)} \ba^{k-1,\ell-1}_{\wt(\vv)}\, e_{\vv},
\end{align*}
where the fourth and fifth equalities follow from induction hypothesis and \eqref{E:4th-reduction} respectively.
But we also have $f_{\upa{(k,\ell^s)}{m}} = \sum_{\t \in \SSTab((k,\ell^s);m)} \ba^{k,\ell}_{\wt(\t)}\, e_{\t}$, so that
$$
\pi_{\mathfrak{T}^{\STab}_{k,\ell,s}}(f_{\upa{(k,\ell^s)}{m}}) = \sum_{\t \in \mathfrak{T}^{\SSTab}_{k,\ell, s}} \ba^{k,\ell}_{\wt(\t)}\, e_{\Xi_{k,\ell,s}(\t)}.$$
Comparing the coefficient of $e_{\s^{k-1,\ell-1}_{\bw_s}}$ in $\pi_{\mathfrak{T}^{\STab}_{k,\ell,s}}(f_{\upa{(k,\ell^s)}{m}})$ for each $\bw_s \in \BW_s$, we get $\ba^{k,\ell}_{\bw_s} = \ba^{k-1,\ell-1}_{\bw_s}$ as desired, and our proof is complete.

\end{enumerate}
\end{proof}

In the next result, $\mathbf{e}_j$, for $j \in [1,s]$, denotes the $j$-th standard basis vector for $\ZZ^s$. Note also that $\BW^1_s = \{ \mathbf{0} \} \cup \{ \mathbf{e}_j \mid j \in [1,s] \}$.

\begin{cor} \label{cor:denominator-ell=1}
We have
$$
f_{\upa{(k,1^s)}{m}} = e_{\upa{(k,1^s)}{m}} + \frac{1}{k+s} \sum_{\substack{\s \in \SSTab((k,1^s);m) \setminus \{ \upa{(k,1^s)}{m} \}}} (-1)^{s-|\wt(\s)|} e_{\s},$$
where $|\wt(\s)| = j$ if $\wt(\s) = \mathbf{e}_j$.

In particular, $\dd_{\upa{(k,1^s)}{m}} = k+s$.
\end{cor}

\begin{proof}
By Theorem \ref{thm:4reductions}(3), $\ba^{k,1}_{m,\bw_s} = \ba^{k,1}_{1,\bw_s}$.  By Theorem \ref{thm:mu=(1)}, $\ba^{k,1}_{1,\mathbf{e}_j} = (-1)^{s-j} \frac{1}{k+s}$ for all $i \in [1,s]$ (since $Q(\s^{k,1,1}_{\mathbf{e}_j}) = [j+1,s+1]$) and $\ba^{k,1}_{1,\mathbf{0}} = 1$.  The corollary thus follows.
\end{proof}

\begin{eg}
Let $k= s = 3$, $m = 2$.  The following is a complete list of $\s_{\bw} \in \SSTab((3,1^3);2)$ for each $\bw \in \BW_3^{1} = \{ (0,0,0), (1,0,0), (0,1,0), (0,0,1) \}$:
\[
\begin{array}{c|cccc}
\bw & (0,0,0) & (1,0,0) & (0,1,0) & (0,0,1) \\ \\
\s_{\bw} &
\ytableausetup{notabloids}
\begin{matrix}
\begin{ytableau}
\scriptstyle 1 & \scriptstyle 2 & \scriptstyle 3 & \scriptstyle \textcolor{violet}{7} & \scriptstyle \textcolor{violet}{8} \\
\scriptstyle \textcolor{red}{4} \\
\scriptstyle \textcolor{blue}{5} \\
\scriptstyle \textcolor{teal}{6}
\end{ytableau} \\
\s_0
\end{matrix}
&
\begin{matrix}
\begin{ytableau}
\scriptstyle 1 & \scriptstyle 2 & \scriptstyle 3 & \scriptstyle \textcolor{red}{4} & \scriptstyle \textcolor{violet}{8} \\
\scriptstyle \textcolor{blue}{5} \\
\scriptstyle \textcolor{teal}{6} \\
\scriptstyle \textcolor{violet}{7}
\end{ytableau} \\
\s_1
\end{matrix}
&
\begin{matrix}
\begin{ytableau}
\scriptstyle 1 & \scriptstyle 2 & \scriptstyle 3 & \scriptstyle \textcolor{blue}{5} & \scriptstyle \textcolor{violet}{8} \\
\scriptstyle \textcolor{red}{4} \\
\scriptstyle \textcolor{teal}{6}\\
\scriptstyle \textcolor{violet}{7}
\end{ytableau} \\
\s_2
\end{matrix}
&
\begin{matrix}
\begin{ytableau}
\scriptstyle 1 & \scriptstyle 2 & \scriptstyle 3 & \scriptstyle \textcolor{teal}{6} & \scriptstyle \textcolor{violet}{8} \\
\scriptstyle \textcolor{red}{4} \\
\scriptstyle \textcolor{blue}{5} \\
\scriptstyle \textcolor{violet}{7}
\end{ytableau} \\
\s_3
\end{matrix}
\end{array}
\]
For each $i \in [1,3]$, let $\s_i' = (7,8) \cdot \s_i$, so that $\s_i'$ is the only other standard $ (5,1^3)$-tableau having the same weight (or colour type) as $\s_i$.
By Corollary \ref{cor:denominator-ell=1},
\[
f_{\s_0} = e_{\s_0} + \tfrac{1}{6} e_{\s_1} + \tfrac{1}{6} e_{\s'_1} - \tfrac{1}{6} e_{\s_2} - \tfrac{1}{6} e_{\s'_2} + \tfrac{1}{6} e_{\s_3} + \tfrac{1}{6} e_{\s'_3}
\quad \text{ and } \quad \dd_{\s_0} = 6.
\]
 \end{eg}

\begin{rem} \hfill
\begin{enumerate}
\item Using Corollary \ref{cor:denominator-ell=1} and \cite[Theorem 3.13]{FLT}, we get $\theta_{(k,1^s),(m)} = (k-1)!m!s!$, generalising $\theta_{(1^n),(m)} = (n-1)!m!$ as obtained in \cite[Proposition 4.1]{FLT}.
\item Corollary \ref{cor:denominator-ell=1} shows that it is possible for $\ba^{k,\ell}_{m,\bw} = - \ba^{k,\ell}_{m,\sigma \cdot \bw}$ (cf.\ Theorem \ref{thm:4reductions}(1)).
\end{enumerate}
\end{rem}

\begin{cor} \label{cor:reduction-denominator-k-ell^s-m} Let
\begin{gather*}
\tilde{k} = k-\ell+\max(s,\min(\ell,m)), \quad
\tilde{\ell} = \min(\ell,m),  \quad \tilde{s} = \min(\ell,m,s), \\
\overline{\BW}^{\tilde{\ell}}_s = \{ (w_1,\dotsc, w_s) \in {\BW}^{\tilde{\ell}}_s \mid w_1 \leq \dotsb \leq w_s \}.
\end{gather*}
Then $\tilde{k} \geq \tilde{\ell} \geq \tilde{s}$, and
\begin{align*}
\dd_{\upa{(k,\ell^s)}{m}}
&= \min\{ \kappa \in \ZZ^+ \mid \kappa\, \ba^{k,\ell}_{m,\bw} \in \ZZ,\ \forall \bw \in \overline{\BW}^{\tilde{\ell}}_s \}
= \dd_{\upa{(\tilde{k},\tilde{\ell}^{\tilde{s}})}{\tilde{\ell}}}.
\end{align*}
\end{cor}

\begin{proof}
Clearly,
$\tilde{k} \geq \max(s,\min(\ell,m)) \geq \min (\ell,m)
= \tilde{\ell} \geq \min (\ell,m,s) = \tilde{s}$.

Next, $\overline{\BW}^{\tilde{\ell}}_s$ is a set of orbit representatives of $\BW^{\tilde{\ell}}_s$ under the action of $\sym{s}$.  Thus,
\begin{align*}
\dd_{\upa{(k,\ell^s)}{m}}
&= \min\{ \kappa\in \ZZ^+ \mid \kappa\, \ba^{k,\ell}_{m,\bw} \in \ZZ,\ \forall \bw \in {\BW}^{\tilde{\ell}}_s \} \\
&= \min\{ \kappa\in \ZZ^+ \mid \kappa\, \ba^{k,\ell}_{m,\sigma \cdot \bw} \in \ZZ,\ \forall \bw \in \overline{\BW}^{\tilde{\ell}}_s\ \forall\sigma\in \sym{s} \} \\
&= \min\{ \kappa \in \ZZ^+ \mid \kappa\, \ba^{k,\ell}_{m,\bw} \in \ZZ,\ \forall \bw \in \overline{\BW}^{\tilde{\ell}}_s \}
\end{align*}
by Theorem \ref{thm:4reductions}(1), proving the first equality.

Now,  when $s > \tilde{\ell}$, there is a bijection $\overline{\BW}^{\tilde{\ell}}_{s-1} \to \overline{\BW}^{\tilde{\ell}}_s$ defined by $\bw \mapsto (0,\bw)$.  Thus,
\begin{align*}
\dd_{\upa{(k+1,\ell^{s-1})}{m}}
&= \min\{ \kappa \in \ZZ^+ \mid \kappa\, \ba^{k+1,\ell}_{m,\bw} \in \ZZ,\ \forall \bw \in \overline{\BW}^{\tilde{\ell}}_{s-1} \} \\
&= \min\{ \kappa \in \ZZ^+ \mid \kappa\, \ba^{k,\ell}_{m,(0,\bw)} \in \ZZ,\ \forall \bw \in \overline{\BW}^{\tilde{\ell}}_{s-1} \} \\
&= \min\{ \kappa \in \ZZ^+ \mid \kappa\, \ba^{k,\ell}_{m,\bw'} \in \ZZ,\ \forall \bw' \in \overline{\BW}^{\tilde{\ell}}_s \} \\
&=\dd_{\upa{(k,\ell^{s})}{m}},
\end{align*}
where the second equality follows from Theorem \ref{thm:4reductions}(2).  Iterating this, we get
$$
\dd_{\upa{(k,\ell^{s})}{m}} = \dd_{\upa{(k+1,\ell^{s-1})}{m}} = \dotsb
= \dd_{\upa{(k+s-\tilde{\ell},\ell^{\tilde{\ell}})}{m}}.
$$
Thus, in general, when $s$ may not be larger than $\tilde{\ell}$, we have
$$
\dd_{\upa{(k,\ell^s)}{m}}
= \dd_{\upa{(k',\ell^{\tilde{s}})}{m}},
$$
where $k' = k+\max(s-\tilde{\ell},0) = k- \tilde{\ell} + \max(s, \tilde{\ell})$.  Theorem \ref{thm:4reductions}(3,4) now shows that
\begin{align*}
\dd_{\upa{(k',\ell^{\tilde{s}})}{m}}
&=
\begin{cases}
\dd_{\upa{(k',\ell^{\tilde{s}})}{m-1}}, &\text{if } m > \ell, \\
\dd_{\upa{(k'-1,(\ell-1)^{\tilde{s}})}{m}}, &\text{if } m < \ell;
\end{cases} \\
&=
\begin{cases}
\dd_{\upa{(k',\ell^{\tilde{s}})}{\ell}}, &\text{if } m > \ell, \\
\dd_{\upa{(k'-\ell+m,m^{\tilde{s}})}{m}}, &\text{if } m < \ell;
\end{cases} \\
&= \dd_{\upa{(k'-\ell+\tilde{\ell},{\tilde{\ell}}^{\tilde{s}})}{\tilde{\ell}}}
= \dd_{\upa{(\tilde{k},{\tilde{\ell}}^{\tilde{s}})}{\tilde{\ell}}},
\end{align*}
since $k'- \ell + \tilde{\ell} = k - \ell + \max(s,\tilde{\ell}) = \tilde{k}.$
\end{proof}

\section{Some general reduction results} \label{sec:general}

In this concluding section, we relate $f_{\upI{\lambda}{\nu}}$ and $\dd_{\upI{\lambda}{\nu}}$ to those labelled by smaller partitions.  Together with the results of the last two sections, we will be able to obtain closed formulae for $\dd_{\upI{\lambda}{\nu}}$ in a slightly more general setting than what we have seen earlier.  We are also able to use these results to obtain upper bounds for $\dd_{\upI{\lambda}{\nu}}$ in general.

Our first result relates $\DD{\upI{\lambda}{\nu}}$ to another labelled by smaller partitions.

\begin{prop} \label{prop:general-basis-reduction}
Let $\lambda = (\lambda_1,\dotsc, \lambda_r)$ and $\nu = (\nu_1,\dotsc, \nu_t)$ be partitions with $[\lambda] \subseteq [\nu]$.
\begin{enumerate}
\item Let $\tilde{\nu} = (\nu_1,\dotsc, \nu_{r-1}, \lambda_r)$.  Then $\DD{\upI{\lambda}{\nu}} = \DD{\upI{\lambda}{\tilde{\nu}}}$.

\item \textup{(Row removal)} Suppose that $\lambda_1 = \nu_1$ and $r \geq 2$.  Let $\check{\lambda} = (\lambda_2,\dotsc, \lambda_r)$ and $\check{\nu} = (\nu_2,\dotsc, \nu_t)$.  Then $\DD{\upI{\lambda}{\nu}} = \DD{\upI{\check{\lambda}}{\check{\nu}}}^{+\lambda_1}$.
\end{enumerate}
\end{prop}

\begin{proof} \hfill
\begin{enumerate}
  \item Since $f_{\upI{\lambda}{\tilde{\nu}}} = \DD{\upI{\lambda}{\tilde{\nu}}} f_{\IT^{\tilde{\nu}}}$,
 we have by Corollary \ref{cor:same-relative-positions}(1)
      \begin{align*}
      f_{\upI{\lambda}{\nu}}
      &= f_{(\upI{\lambda}{\tilde{\nu}}) \uparrow^{\nu}} = \DD{\upI{\lambda}{\tilde{\nu}}} f_{\upI{\tilde{\nu}}{\nu}}
      = \DD{\upI{\lambda}{\tilde{\nu}}} f_{\IT^{\nu}}
      = \DD{\upI{\lambda}{\tilde{\nu}}} e_{\IT^{\nu}}
      = \sum_{\substack{\tilde{\s} \in \STab(\tilde{\nu})}} q_{\upI{\lambda}{\tilde{\nu}},\tilde{\s}}\, e_{d(\tilde{\s}) \cdot \IT^{\nu}}.
      \end{align*}
      Since $d(\tilde{\s}) \cdot \IT^{\nu} = \upt{\tilde{\s}}{\nu} \in \STab(\nu)$ for all $\tilde{\s} \in \STab(\tilde{\nu})$, and the map $\tilde{\s} \mapsto \upt{\tilde{\s}}{\nu}$ is injective, the desired result follows.

  \item Since $f_{d(\upI{\check{\lambda}}{\check{\nu}}) \cdot \IT^{\check{\nu}}} = f_{\upI{\check{\lambda}}{\check{\nu}}} = \DD{\upI{\check{\lambda}}{\check{\nu}}} f_{\IT^{\check{\nu}}}$,
      and $\lambda_1+\upI{\check{\lambda}}{\check{\nu}} (i,j) = \upI{\lambda}{\nu}(i+1,j)$ for all $(i,j) \in [\check{\nu}]$ so that $\res_{\upI{\check{\lambda}}{\check{\nu}}}(i) = 1+\res_{\upI{\lambda}{\nu}}(i+\lambda_1)$ for all $i \in [1,|\check{\nu}|]$, we have by Proposition \ref{prop:same-relative-positions}(2) (with $z= \lambda_1$)
      $$
      f_{\upI{\lambda}{\nu}}
      = f_{d(\upt{\check{\lambda}}{\check{\nu}})^{+\lambda_1} \cdot \IT^{\nu}}
      = \DD{\upt{\check{\lambda}}{\check{\nu}}}^{+\lambda_1} f_{\IT^{\nu}}
      = \sum_{\substack{\check{\s} \in \STab(\check{\nu}) }} q_{\upI{\check{\lambda}}{\check{\nu}},\check{\s}}\, e_{d(\check{\s})^{+\lambda_1} \cdot \IT^{\nu}}.
      $$
      Since $d(\check{\s})^{+\lambda_1} \cdot \IT^{\nu} \in \STab(\nu)$ for all $\check{\s} \in \STab(\tilde{\nu})$, and the map $\check{\s} \mapsto d(\check{\s})^{+\lambda_1} \cdot \IT^{\nu}$ is injective, the desired result follows.
\end{enumerate}
\end{proof}

\begin{cor}[cf.\ {\cite[Theorem 1]{RHansen10}}] \label{cor:remove-one-node}
Let $\nu$ be a partition and let $\lambda$ be the partition obtained from $\nu$ by removing a removable node, say on its $i$-th row.  Then
$$
\DD{\upI{\lambda}{\nu}} = \DD{\upa{(\lambda_i,\dotsc, \lambda_r)}{1}}^{+(\sum_{j=1}^{i-1} \lambda_j)}
= \sum_{\s \in \SSTab((\lambda_i,\dotsc,\lambda_r);1)} a_{\s} d(\s)^{+(\sum_{j=1}^{i-1} \lambda_j)}.
$$
(See Theorem \ref{thm:mu=(1)} for the definition of $a_{\s}$.)
\end{cor}

\begin{proof}
This follows by iterating Proposition \ref{prop:general-basis-reduction}(2) and Theorem \ref{thm:mu=(1)}.
\end{proof}

We next relate $\dd_{\upI{\lambda}{\nu}}$ to another labelled by smaller partitions.

\begin{thm} \label{thm:general}
Let $\lambda= (\lambda_1,\dotsc, \lambda_r)$ and $\nu = (\nu_1,\dotsc, \nu_t)$ be partitions with $[\lambda] \subseteq [\nu]$.
\begin{enumerate}
  \item We have $\dd_{\upI{\lambda}{\nu}} = \dd_{\upI{\lambda}{(\nu_1,\dotsc,\nu_{r-1},\lambda_r)}}$.
  \item If $\lambda_1 = \nu_1$ and $r \geq 2$, then $\dd_{\upI{\lambda}{\nu}} = \dd_{\upI{(\lambda_2,\dotsc,\lambda_r)}{(\nu_2,\dotsc,\nu_t)}}$.

  \item For $m \in [1,\nu_1-\lambda_1]$, we have
  $$\dd_{\upI{\lambda}{\nu}} \mid \dd_{\upa{\lambda}{m}} \dd_{\upI{\lambda+(m)}{\nu}}.$$

  \item For $i \in [2,r-1]$ and $m \in \ZZ^+$, we have
  \begin{align*}
  \dd_{\upa{\lambda}{m}} &\mid \dd_{\upa{(\lambda_1+i-1,\lambda_{i+1},\dotsc, \lambda_r)}{m}} \dd_{\upa{(\lambda_1,\dotsc,\lambda_i)}{m}} .
  \end{align*}
\end{enumerate}
\end{thm}

\begin{proof}
Parts (1) and (2) follow from Proposition \ref{prop:general-basis-reduction}.

For part (3)
  we have $f_{\upI{\lambda}{\nu}} = \DD{\upa{\lambda}{m}} \DD{\upI{\lambda+(m)}{\nu}}\, e_{\IT^{\nu}}$ for any $m \in [1,\nu_1-\lambda_1]$ by Corollary \ref{cor:same-relative-positions}(2).  Thus,
      $$
      \dd_{\upa{\lambda}{m}} \dd_{\upI{\lambda+(m)}{\nu}} f_{\upI{\lambda}{\nu}}
      = (\dd_{\upa{\lambda}{m}}\DD{\upa{\lambda}{m}}) (\dd_{\upI{\lambda+(m)}{\nu}}\DD{\upI{\lambda+(m)}{\nu}})\, e_{\IT^{\nu}} \in (\ZZ \sym{|\nu|}) S^{\ZZ}_{\nu} = S^{\ZZ}_{\nu},
      $$
      so that $\dd_{\upI{\lambda}{\nu}} \mid \dd_{\upa{\lambda}{m}} \dd_{(\IT^{\lambda+(m)}) \uparrow^{\nu}}$ as desired.

Part (4) uses Corollary \ref{cor:same-relative-positions}(3) and an argument similar to part (3).
\end{proof}

By iterating Theorem \ref{thm:general} together with Corollaries \ref{cor:mu=(1)}, \ref{cor:denominator-klm} and \ref{cor:reduction-denominator-k-ell^s-m}, we can obtain (possibly many) upper bounds for any $\dd_{\upI{\lambda}{\nu}}$.
For example, when combining part (3) with part (2) in Theorem \ref{thm:general}, we get
$$
\dd_{\upI{\lambda}{\nu}} \mid \dd_{\upa{\lambda}{\nu_1-\lambda_1}} \dd_{\upI{(\lambda_2,\dotsc,\lambda_r)}{(\nu_2,\dotsc, \nu_{t})}}
\mid \dd_{\upa{\lambda}{\nu_1-\lambda_1}} \dd_{\upa{(\lambda_2,\dotsc, \lambda_r)}{\nu_2-\lambda_2}} \dd_{\upI{(\lambda_3,\dotsc,\lambda_r)}{(\nu_3,\dotsc, \nu_{t})}} \mid \dotsb.
$$
We may obtain upper bounds for $\dd_{\upa{\lambda}{\nu_1-\lambda_1}}, \dd_{\upa{(\lambda_2,\dotsc, \lambda_r)}{\nu_2-\lambda_2}}, \dotsc$ further by using part (3) or (4) of Theorem \ref{thm:general}.

We demonstrate this process of obtaining upper bounds below with the example of $\dd_{\upa{(k,\ell^s)}{\ell}}$ where $k \geq \ell \geq s$.  Recall that in the last section, we showed that all denominators of the form $\dd_{\upa{(k,\ell^s)}{m}}$ can be reduced to this form (Corollary \ref{cor:reduction-denominator-k-ell^s-m}).

\begin{prop} \label{prop:upper-bound}
Let $k,\ell,s \in \ZZ^+$ with $k \geq \ell \geq s$.  Then
$$
\dd_{\upa{(k,\ell^s)}{\ell}} \mid \gcd\left(\prod_{i=1}^{\ell} (k-\ell+s+i), \prod_{j=1}^{s} \frac{\lcm[k-\ell+j,k+j]}{k-\ell+j}\right).
$$
\end{prop}

\begin{proof}
Iterating Theorem \ref{thm:general}(3),
we get
\begin{align*}
\dd_{\upa{(k,\ell^s)}{\ell}} \mid \dd_{\upa{(k,\ell^s)}{1}} \dd_{\upa{(k+1,\ell^s)}{\ell-1}} \mid \dotsb
&\mid  \dd_{\upa{(k,\ell^s)}{1}} \dd_{\upa{(k+1,\ell^s)}{1}} \dotsb \dd_{\upa{(k+\ell-1,\ell^s)}{1}} \\
& = \prod_{i=1}^{\ell} (k-\ell+s+i)
\end{align*}
by Corollary \ref{cor:mu=(1)}.  We may also iterate Theorem \ref{thm:general}(4) with $i=2$ and get
\begin{align*}
\dd_{\upa{(k,\ell^s)}{\ell}}
\mid \dd_{\upa{(k,\ell)}{\ell}} \dd_{\upa{(k+1,\ell^{s-1})}{\ell}} \mid \dotsb
&\mid \dd_{\upa{(k,\ell)}{\ell}} \dd_{\upa{(k+1,\ell)}{\ell}} \dotsb \dd_{\upa{(k+s-1,\ell)}{\ell}} \\
&= \prod_{j=1}^{s} \tfrac{\lcm[k-\ell+j,\,k+j]}{k-\ell+j}
\end{align*}
by Corollary \ref{cor:denominator-klm}.  The proposition thus follows.
\end{proof}

\begin{eg}
In this example, we illustrate how Proposition \ref{prop:upper-bound} may be used to show that $\dd_{\upa{(k,2^2)}{2}} = (k+1)(k+2)$ for all $k \geq 2$.  By Corollary \ref{cor:reduction-denominator-k-ell^s-m}, this also gives
\[
\dd_{\upa{(k+\ell-s,\ell^s)}{m}} = (k+1)(k+2)
\]
whenever $k \geq s \geq 2$ and $\min(\ell,m) = 2$.

By Corollary \ref{cor:same-relative-positions}(2) and Theorem \ref{thm:mu=(1)}, we have
\begin{align*}
f_{\upa{(k,2^2)}{2}} &= \DD{\upa{(k,2^2)}{1}} \DD{\upa{(k+1,2^2)}{1}} e_{\IT^{(k+2,2^2)}} \\
&= \sum_{\substack{\uu \in \SSTab((k,2^2);1) \\ \vv \in \SSTab((k+1,2^2);1)}} a_{\uu} a_{\vv} e_{d(\uu) \cdot \vv}.
\end{align*}
For each $\vv \in \SSTab((k+1,2^2);1)$, let $b_{\vv} \in [k+2,k+6]$ such that $\vv$ contains $[1,k+1]$ and $b_{\vv}$ in its first row.
Let
$$
\s = \s^{k,2,2}_{(1,1)} = \raisebox{3mm}{
\ytableausetup{mathmode, boxsize=1.8em}
\begin{ytableau}
\scriptstyle 1 & \scriptstyle 2  & \cdots & \cdots & \scriptstyle k & \scriptstyle \FM{k+2} & \scriptstyle \KJ{k+4} \\
\scriptstyle \FM{k+1} & \scriptstyle \KJ{k+3} \\
\scriptstyle \KM{k+5} & \scriptstyle \KM{k+6}
\end{ytableau} }.
$$
By Proposition \ref{prop:Garnir}(3) (with $l=1$), $e_{d(\uu) \cdot \vv}$ does not contribute to the coefficient of $e_{\s}$ in $f_{\upa{(k,2^2)}{2}}$ unless $d(\uu) \cdot \vv$
contains $[1,k]\cup\{k+2,k+4 \}$ in its first row, or equivalently, $d(\uu) (\{k+1,b_{\vv} \}) = \{k+2,k+4\}$. Assume thus $d(\uu) (\{k+1,b_{\vv} \}) = \{k+2,k+4\}$.

\begin{description}[leftmargin=2em]
\item[Case 1a. $d(\uu) (k+1) = k+2$, $d(\uu)(b_\vv) = k+4$]
There are exactly two such $\uu \in \SSTab((k,2^2);1)$, namely
$$
\uu_1 = \raisebox{3mm}{
\ytableausetup{mathmode, boxsize=1.8em}
\begin{ytableau}
\scriptstyle 1 & \scriptstyle 2  & \cdots & \cdots & \scriptstyle k & \scriptstyle \FM{k+2} \\
\scriptstyle \FM{k+1} & \scriptstyle \KJ{k+3} \\
\scriptstyle \KJ{k+4} & \scriptstyle \KM{k+5}
\end{ytableau} }, \quad
\uu_2 = \raisebox{3mm}{
\ytableausetup{mathmode, boxsize=1.8em}
\begin{ytableau}
\scriptstyle 1 & \scriptstyle 2  & \cdots & \cdots & \scriptstyle k & \scriptstyle \FM{k+2} \\
\scriptstyle \FM{k+1} & \scriptstyle \KJ{k+4} \\
\scriptstyle \KJ{k+3} & \scriptstyle \KM{k+5}
\end{ytableau} }.
$$
We have $d(\uu_1) = (k+1,k+2)$, $d(\uu_2) = (k+1,k+2)(k+3,k+4)$, $a_{\uu_1} = -\frac{1}{k+1} = a_{\uu_2}$, and
$b_{\vv} =
\begin{cases}
k+4, &\text{if }\uu = \uu_1; \\
k+3, &\text{if }\uu = \uu_2. \end{cases}$

\item[Case 1b. $d(\uu) (k+1) = k+4$, $d(\uu)(b_\vv) = k+2$]
There is only one such $\uu \in \SSTab((k,2^2);1)$, namely
$$
\uu_3 = \raisebox{3mm}{
\ytableausetup{mathmode, boxsize=1.8em}
\begin{ytableau}
\scriptstyle 1 & \scriptstyle 2  & \cdots & \cdots & \scriptstyle k & \scriptstyle \KJ{k+4} \\
\scriptstyle \FM{k+1} & \scriptstyle \FM{k+2} \\
\scriptstyle \KJ{k+3} & \scriptstyle \KM{k+5}
\end{ytableau}\ }.
$$
We have $d(\uu_3) = (k+1,k+4,k+3,k+2)$, $a_{\uu_3} = \frac{1}{k+1}$ and $b_{\vv} = k+3$.
\end{description}
Thus $b_{\vv} \in \{k+3,k+4\}$.  We now look at these $\vv \in \SSTab((k+1,2^2);1)$.
\begin{description}[leftmargin=2em]
\item[Case 2a. $b_{\vv} = k+4$]
There is only one such $\vv \in \SSTab((k+1,2^2);1)$, namely
$$
\vv_1 = \raisebox{3mm}{
\ytableausetup{mathmode, boxsize=1.8em}
\begin{ytableau}
\scriptstyle 1 & \scriptstyle 2  & \cdots & \cdots & \scriptstyle k & \scriptstyle k+1 & \scriptstyle \KJ{k+4} \\
\scriptstyle \FM{k+2} & \scriptstyle \FM{k+3} \\
\scriptstyle \KJ{k+5} & \scriptstyle \KM{k+6}
\end{ytableau} }.
$$
We have $a_{\vv_1} = \frac{1}{k+2}$.

\item[Case 2b. $b_{\vv} = k+3$]
There are exactly two such $\vv \in \SSTab((k+1,2^2);1)$, namely
$$
\vv_2 = \raisebox{3mm}{
\ytableausetup{mathmode, boxsize=1.8em}
\begin{ytableau}
\scriptstyle 1 & \scriptstyle 2  & \cdots & \cdots & \scriptstyle k & \scriptstyle k+1 & \scriptstyle \FM{k+3} \\
\scriptstyle \FM{k+2} & \scriptstyle \KJ{k+4} \\
\scriptstyle \KJ{k+5} & \scriptstyle \KM{k+6}
\end{ytableau} }, \quad
\vv_3 = \raisebox{3mm}{
\ytableausetup{mathmode, boxsize=1.8em}
\begin{ytableau}
\scriptstyle 1 & \scriptstyle 2  & \cdots & \cdots & \scriptstyle k & \scriptstyle k+1 & \scriptstyle \FM{k+3} \\
\scriptstyle \FM{k+2} & \scriptstyle \KJ{k+5} \\
\scriptstyle \KJ{k+4} & \scriptstyle \KM{k+6}
\end{ytableau} }.
$$
We have $a_{\vv_2} = -\frac{1}{k+2} = a_{\vv_3}$.
\end{description}

Now, $d(\uu_1) \cdot \vv_1 = d(\uu_2) \cdot \vv_2 = \overline{d(\uu_3) \cdot \vv_2} = \s$ for all $i \in [1,3]$, while
$$
d(\uu_2) \cdot \vv_3 = \overline{d(\uu_3) \cdot \vv_3} = \raisebox{3mm}{
\ytableausetup{mathmode, boxsize=1.8em}
\begin{ytableau}
\scriptstyle 1 & \scriptstyle 2  & \cdots & \cdots & \scriptstyle k & \scriptstyle k+2 & \scriptstyle k+4 \\
\scriptstyle k+1 & \scriptstyle k+5 \\
\scriptstyle k+3 & \scriptstyle k+6
\end{ytableau} } \in \STab((k+2,2^2)) \setminus \{\s\}.
$$
Thus the coefficient of $e_{\s}$ in $f_{\upa{(k,2^2)}{2}}$ equals
$$
a_{\uu_1}a_{\vv_1} + a_{\uu_2}a_{\vv_2} + a_{\uu_3}a_{\vv_2} = (-\tfrac{1}{k+1})(\tfrac{1}{k+2}) + (-\tfrac{1}{k+1})(-\tfrac{1}{k+2}) + (\tfrac{1}{k+1})(-\tfrac{1}{k+2}) = -\tfrac{1}{(k+1)(k+2)}.
$$
This yields
$$
(k+1)(k+2) \mid \dd_{\upa{(k,2^2)}{2}} \mid \gcd( \prod_{i=1}^2 (k+i), \prod_{j=1}^2 \tfrac{\lcm[k-2+j,k+j]}{k-2+j}) \mid \prod_{i=1}^2 (k+i) = (k+1)(k+2)
$$
by Proposition \ref{prop:upper-bound}, forcing equality throughout.
\end{eg}

We now give an indication how Theorem \ref{thm:summary} comes about:

\begin{proof}[Proof of Theorem \ref{thm:summary}]
Part (1) follows from Corollary \ref{cor:mu=(1)} and Theorem \ref{thm:general}(2). Part (2) follows from Corollary \ref{cor:denominator-klm} and Theorem \ref{thm:general}(1). Part (3) follows from Corollary \ref{cor:reduction-denominator-k-ell^s-m} and Theorem \ref{thm:general}(1), while parts (4)--(6) follow from Theorem \ref{thm:general}.
\end{proof}

We end the paper with the following concluding remark.

\begin{rem}
Let $\lambda \vdash n$.  Following \cite{RHansen10}, we can provide an estimate for $\DD{\s}$ and an upper bound for $\dd_{\s}$ for a general $\s \in \STab(\lambda)$ as follows.
For each $i\in [1,n]$, let $\s_i = \s{\downarrow_i}$, and let $\lambda^{i} = \Shape(\s_i)$.
Then $\upt{\s_i}{\lambda^{i+1}} = \s_{i+1}$ for all $i \in [1,n-1]$.
By iterating Corollary \ref{cor:same-relative-positions}(1), we get $f_{\s} = \DD{\upI{\lambda^{1}}{\lambda^{2}}} \dotsm \DD{\upI{\lambda^{n-1}}{\lambda^{n}}}\, e_{\IT^{\lambda}}$ .
Thus, $\DD{\upI{\lambda^{1}}{\lambda^{2}}} \dotsm \DD{\upI{\lambda^{n-1}}{\lambda^{n}}}$ equals $\DD{\s}$ modulo the annihilator of $e_{\IT^{\lambda}}$, so that the former is an estimate of the latter, and $\dd_{\s} \mid \prod_{i=1}^{n-1} \dd_{\upI{\lambda^{i}}{\lambda^{i+1}}}$.
Note that Corollary \ref{cor:remove-one-node} and Theorem \ref{thm:summary}(1) give closed formulae for $\DD{\upI{\lambda^{i}}{\lambda^{i+1}}}$ and $\dd_{\upI{\lambda^{i}}{\lambda^{i+1}}}$ respectively for each $i \in [1,n-1]$.
\end{rem}

\end{document}